\documentclass[11pt,a4paper]{amsart}
\usepackage {amsmath, amssymb, a4wide, epsfig, enumerate, psfrag}
\usepackage[latin1]{inputenc}
\usepackage[english]{babel}
\usepackage[active]{srcltx}

\usepackage[title,toc,page,header]{appendix}

\usepackage{centernot}

\date{\today}

\author{Nicolas Tholozan}
\address{Laboratoire J.A. Dieudonn\'e \\
UMR CNRS 7351 \\
Universit\'e de Nice Sophia-Antipolis \\
Parc Valrose \\
06108 NICE Cedex 02\\
FRANCE}
\email{tholozan@unice.fr}

\thanks{The author acknowledges support from U.S. National Science Foundation grants DMS 1107452, 1107263, 1107367 "RNMS: Geometric Structures and Representation Varieties" (the GEAR Network).}

\title[Surface group representations and AdS $3$-manifolds]{Dominating surface group representations and deforming closed anti-de Sitter $3$-manifolds}

\newcommand{\R}{\mathbb{R}}

\newcommand{\norm}[1]{\left\Vert#1\right\Vert}

 \newcommand{\anglehyp}[1]{\angle_{hyp}(#1)}
 
 \newcommand{\axe}{\mathrm{axis}}

 \newcommand{\tend}[1]{\underset{#1}{\longrightarrow}}

\newcommand{\Teich}{\mathcal{T}}
\newcommand{\Adm}{\mathrm{Adm}}
\newcommand{\Dom}{\mathrm{Dom}}

\newcommand{\C}{\mathbb{C}}
\newcommand{\N}{\mathbb{N}}

\renewcommand{\H}{\mathbb{H}}

\newcommand{\Isom}{\mathrm{Isom}}

\newcommand{\Hom}{\mathrm{Hom}}

\newcommand{\Id}{\mathrm{Id}}

\newcommand{\Rep}{\mathrm{Rep}}

\newcommand{\CAT}{\mathrm{CAT}}

\newcommand{\Tr}{\mathrm{Tr}}

\newcommand{\scal}[2]{\left \langle #1, #2 \right \rangle}

\renewcommand{\d}{\mathrm{d}}

\newcommand{\function}[5]{
\begin{array}{rrcl}
#1 : & #2 & \to & #3 \\
\ & #4 & \mapsto & #5
\end{array}
}

\renewcommand{\epsilon}{\varepsilon}

\newcommand{\1}{\mathbf{1}}

\newcommand{\T}{\mathrm{T}}

\newcommand{\PSL}{\mathrm{PSL}}

\newcommand{\PSO}{\mathrm{PSO}}

\renewcommand{\phi}{\varphi}

\newcommand{\Lip}{\mathrm{Lip}}

\DeclareMathOperator{\im}{Im}

\newcommand{\Vol}{\mathrm{Vol}}

\renewcommand{\tilde}[1]{\widetilde{#1}}
\newcommand{\PPhi}{\mathbf{\Phi}}
\newcommand{\PPsi}{\mathbf{\Psi}}
\newcommand{\E}{\mathbf{E}}
\newcommand{\Grad}{\mathbf{Grad}}
\newcommand{\QD}{\mathrm{QD}}
\renewcommand{\T}{\mathrm{T}}
\newcommand{\F}{\mathbf{F}}
\renewcommand{\Im}{\mathrm{Im}}
\renewcommand{\Re}{\mathrm{Re}}

\newcommand{\Z}{\mathbb{Z}}

\theoremstyle{prop}
\newtheorem{prop} {Proposition} [section]
\newtheorem*{propstar}{Proposition}
\newtheorem{thm}[prop] {Theorem}
\newtheorem{lem}[prop]{Lemma}
\newtheorem{coro}[prop]{Corollary}

\theoremstyle{definition}
\newtheorem{definition}[prop]{Definition}

\theoremstyle{remark}

\newtheorem{rmk}[prop]{Remark}

\theoremstyle{theorem}
\newtheorem*{CiteThm}{Theorem}

\newtheorem{BigThm}{Theorem}
\newtheorem{BigCoro}[BigThm]{Corollary}

\begin{document}

\begin{abstract}
Let $S$ be a closed oriented surface of negative Euler characteristic and $M$ a complete contractible Riemannian manifold. A Fuchsian representation $j: \pi_1(S) \to \Isom^+(\H^2)$ \emph{strictly dominates} a representation $\rho:\pi_1(S)\to \Isom(M)$ if there exists a $(j, \rho)$-equivariant map from $\H^2$ to $M$ that is $\lambda$-Lipschitz for some $\lambda <1$. In a previous paper by Deroin--Tholozan, the authors construct a map $\PPsi_\rho$ from the Teichm\"uller space $\Teich(S)$ of the surface $S$ to itself and prove that, when $M$ has sectional curvature $\leq -1$, the image of $\PPsi_\rho$ lies (almost always) in the domain $\Dom(\rho)$ of Fuchsian representations stricly dominating $\rho$. Here we prove that $\PPsi_\rho: \Teich(S) \to \Dom(\rho)$ is a homeomorphism. As a consequence, we are able to describe the topology of the space of pairs of representations $(j,\rho)$ from $\pi_1(S)$ to $\Isom^+ (\H^2)$ with $j$ Fuchsian strictly dominating $\rho$. In particular, we obtain that its connected components are classified by the Euler class of $\rho$. The link with anti-de Sitter geometry comes from a theorem of Kassel stating that those pairs parametrize deformation spaces of anti-de Sitter structures on closed $3$-manifolds.
\end{abstract} 

\maketitle

\section*{Introduction}

\subsection{Closed AdS $3$-manifolds}

An \emph{anti-de Sitter} (AdS) manifold is a smooth manifold equipped with a Lorentz metric of constant negative sectional curvature. In dimension $3$, those manifolds are locally modelled on $\PSL(2,\R)$ with its Killing metric, whose isometry group identifies to a $\Z/2\Z \times \Z/ 2\Z$ extension of $\PSL(2,\R) \times \PSL(2,\R)$ acting by left and right translations, i.e.
$$(g_1,g_2)\cdot x = g_1 x g_2^{-1}$$
for all $(g_1,g_2) \in \PSL(2,\R) \times \PSL(2,\R)$ and all $x\in \PSL(2,\R)$.

The specific interest for anti-de Sitter structures on $3$-manifolds takes its roots in the geome\-trization program of Thurston. Following Thurston's classification of the eight Riemannian geometries in dimension $3$, Scott \cite{Scott83} proved that a large class of closed $3$-manifolds could be modelled on the sixth geometry, that is, $\tilde{\PSL(2,\R)}$ with a left invariant Riemannian metric that is also $\tilde{\PSO(2)}$-invariant on the right. In particular, those manifolds carry an AdS structure, called standard. At first, it was conjectured that all closed AdS $3$-manifolds had to be standard (the conjecture still holds in higher dimension, see \cite{Zeghib98}). However, it turned out progressively that the space of AdS structures on closed $3$-manifolds is much richer than the space of standard structures.

First, Goldman \cite{Goldman85} noticed that, in some cases, a standard structure could be deformed into a non-standard one. Namely, he proved that, if $S$ is a closed surface of negative Euler characteristic, $j: \pi_1(S) \to \PSL(2,\R)$ a discrete and faithful representation and $\rho: \pi_1(S) \to \PSL(2,\R)$ a representation with values in a $1$-parameter subgroup that is sufficiently close to the trivial representation, then the embedding
$$(j,\rho)(\pi_1(S)) = \{ (j(\gamma), \rho(\gamma)), \gamma \in \pi_1(S) \} \subset \PSL(2,\R) \times \PSL(2,\R)$$
acts properly discontinuously and cocompactly on $\PSL(2,\R)$. (This was generalized later by Kobayashi \cite{Kobayashi98}.) The quotient is nonstandard whenever the image of $\rho$ is not included in a compact subgroup.

Recall that a Lorentz metric on a manifold possesses a geodesic flow, and that the metric is called \emph{complete} if its geodesics run for all time. In dimension $3$, a complete AdS manifold is a quotient of $\tilde{\PSL(2,\R)}$ by a subgroup of $\tilde{\PSL(2,\R)} \times \tilde{\PSL(2,\R)}$ acting properly discontinuously on $\tilde{\PSL(2,\R)}$. In \cite{KulkarniRaymond85}, Kulkarni and Raymond addressed the question of describing all complete AdS $3$-manifolds.  Their first assertion is that complete compact AdS $3$-manifolds have \emph{finite level}, that is, they are actually quotients of a finite cover of $\PSL(2,\R)$. (A correct proof was given by Salein in his thesis \cite{SaleinThese}.) Therefore, closed AdS $3$-manifolds are cyclic coverings of a quotient of $\PSL(2,\R)$ by the action of a discrete subgroup of $\PSL(2,\R) \times \PSL(2,\R)$.
Kulkarni and Raymond then proved that all such quotients look like Goldman's examples in a certain sense.
\begin{CiteThm}[Kulkarni--Raymond, \cite{KulkarniRaymond85}]
Let $\Gamma$ be a discrete group acting properly discontinuously and cocompactly on $\PSL(2,\R)$ via a faithful representation $h: \Gamma \to \PSL(2,\R)\times \PSL(2,\R)$. Assume moreover that $\Gamma$ is torsion-free. Then $\Gamma$ is the fundamental group of a closed surface of negative Euler characteristic and
$$h = (j,\rho)~,$$
with either $j$ or $\rho$ discrete and faithful.
\end{CiteThm}

\begin{rmk}
This theorem was generalized by Kobayashi \cite{Kobayashi93} and Kassel \cite{Kassel08} to quotients of any rank $1$ semisimple Lie group $G$ under the action of some subgroup of $G\times G$ by left and right translations.
\end{rmk}

\begin{rmk}
Contrary to the Riemannian setting, a Lorentz metric on a closed manifold may not be complete. However, Klingler proved later \cite{Klingler96}, generalizing Carri\`ere's theorem \cite{Carriere89}, that closed Lorentz manifolds of constant curvature are complete. Hence the completeness assumption in Kulkarni--Raymond's work can be removed.
\end{rmk}

Recall that, as an application of Selberg's lemma \cite{Selberg89}, if there exists a faithful representation of $\Gamma$ into $\PSL(2,\R) \times \PSL(2,\R)$, then $\Gamma$ admits a torsion-free finite index subgroup. Therefore, Kulkarni--Raymond's theorem describes all closed AdS $3$-manifolds up to a finite cover. A consequence of their theorem is that any closed AdS $3$-manifold is finitely covered by a circle bundle over a surface, and is itself a Seifert fiber space.

We now fix a closed oriented surface $S$ of negative Euler characteristic $\chi(S)$. A given pair $(j,\rho)$ of representations of $\pi_1(S)$ into $\PSL(2,\R)$ does not necessarily induce a proper action of $\pi_1(S)$ on $\PSL(2,\R)$. For instance, if $\rho(\gamma)$ is conjugate to $j(\gamma)$ for some $\gamma \neq \1$, then the infinite subgroup generated by $\gamma$ acts via $(j,\rho)$ with a fixed point. A pair $(j,\rho)$ is called \emph{admissible} if $(j,\rho)(\pi_1(S))$ acts properly discontinuously and cocompactly on $\PSL(2,\R)$. 

Let $(j,\rho)$ be an admissible pair. One easily checks that $(\rho,j)$ is also admissible. If $\sigma$ denotes the outer automorphism of $\PSL(2,\R)$ given by conjugation by 
\[ \left(\begin{array}{cc}
1 & 0 \\
0 & -1 \end{array}
\right)~,\]
 then $(\sigma \circ j, \sigma \circ \rho)$ is also admissible. Finally, for any $g_1, g_2 \in \PSL(2,\R)$, the pair \mbox{$(C_{g_1} \circ j, C_{g_2} \circ \rho)$} is still admissible (where $C_{g_i}$ denotes the conjugation by $g_i$ on $\PSL(2,\R)$). We will denote by $\Adm(S)$ the quotient of the space of admissible pairs of representations of $\pi_1(S)$ by these transformations. The main objective of this article is to describe the topology of $\Adm(S)$ and, in particular, classify its connected components. This will answer one of the many interesting questions of the recent survey: \emph{Some open questions in anti-de Sitter geometry} (\cite{QuestionsAdS}, section 2.3).

Recall that, by the work of Goldman \cite{GoldmanThese}, connected components of the space of representations of $\pi_1(S)$ into $\PSL(2,\R)$ are classified by an integer called the \emph{Euler class}. It can take any value between $\chi(S)$ and $-\chi(S)$ and is extremal (i.e.\ equal to $\pm \chi(S)$) if and only if the representation is \emph{Fuchsian}, that is, discrete and faithful. The postcomposition with the outer automorphism $\sigma$ exchanges the components of extremal Euler class. Therefore, up to switching the factors and postcomposing with $\sigma$, we can assume that any admissible pair $(j,\rho)$ has $j$ Fuchsian of Euler class $-\chi(S)$. From now on, we will consider $\Adm(S)$ as the space of admissible pairs $(j,\rho)$ with $j$ of Euler class $-\chi(S)$ modulo the action of $\PSL(2,\R) \times \PSL(2,\R)$ by conjugation.

If $j$ is Fuchsian and $\rho$ takes values into a compact subgroup of $\PSL(2,\R)$, then the pair $(j,\rho)$ is clearly admissible, and the quotient is standard. According to \cite{Goldman85}, there are some non-standard admissible pairs  with $\rho$ in the connected component of the trivial representation. Kulkarni and Raymond believed that all non-standard structures could be obtained by deforming standard ones. It was proven wrong by Salein.

Let us denote by $\H^2$ the Poincar\'e half-plane and by $g_P$ the Poincar\'e metric on $\H^2$, of curvature~$-1$. The group $\PSL(2,\R)$ acts on $\H^2$ by homographies and identifies with the group of orientation preserving isometries $\Isom^+(\H^2,g_P)$.
Salein noticed \cite{Salein00} that a sufficient condition for a pair $(j,\rho)$ to be admissible is that $j$ \emph{stricty dominates} $\rho$, in the sense that there exists a $(j,\rho)$-equivariant map from $(\H^2,g_P)$ to $(\H^2,g_P)$ which is $\lambda$-Lipschitz for some $\lambda < 1$. As a consequence, he constructed admissible pairs $(j,\rho)$ with $\rho$ of any non-extremal Euler class. In particular, those admissible pairs cannot be continuously deformed into standard ones. 
Lastly, Kassel proved in~\cite{KasselThese} that Salein's sufficient condition for admissibility is also necessary:

\begin{CiteThm}[Kassel]
Let $S$ be a closed surface of negative Euler characteristic and $j,\rho$ two representations of $\pi_1(S)$ into $\PSL(2,\R)$, with $j$ Fuchsian. Then the pair $(j,\rho)$ is admissible if an only if there exists a $(j,\rho)$-equivariant map from $(\H^2,g_P)$ to $(\H^2,g_P)$ that is $\lambda$-Lipschitz for some $\lambda <1$.
\end{CiteThm}

\subsection{Dominated representations}
Kassel's criterion for admissibility raises many questions that may turn out to be interesting beyond the scope of $3$-dimensional AdS geometry. Indeed, one can extend the definition of domination to a more general setting. Consider a closed surface $S$ of negative Euler characteristic, and a contractible Riemannian manifold $(M,g_M)$. Let $j: \pi_1(S) \to \PSL(2,\R)$ be a Fuchsian representation, and $\rho$ a representation of $\pi_1(S)$ into $\Isom(M,g_M)$. Since $M$ is contractible, there always exists a smooth $(j,\rho)$-equivariant map from $\H^2$ to $M$. Since $S$ is compact, this map is Lipschitz. We can thus define the \emph{minimal Lipschitz constant}
\[\Lip(j,\rho) = \inf \left\lbrace \lambda \in \R_+ \mid \exists f: \H^2 \to M \textrm{ $(j,\rho)$-equivariant and $\lambda$-lispchitz} \right\rbrace~.\]
We then say that $j$ strictly dominates $\rho$ if $\Lip(j,\rho) < 1$.

\subsubsection{Length spectrum} \label{ss:LengthSpectrum}
A related notion of domination comes from the comparison of the length spetra of $j$ and $\rho$. Recall that if $g$ is an isometry of a metric space $(M,d)$, the \emph{translation length} of $g$ is the number
$$l(g) := \inf_{x\in M} d(x,g\cdot x)~.$$
The \emph{length spectrum} of a representation $\rho: \pi_1(S) \to \Isom(M)$ is the function that associates to $\gamma \in \pi_1(S)$ the translation length of $\rho(\gamma)$. We say that the length spectrum of $j$ strictly dominates the length spectrum of $\rho$ if there exists a constant $\lambda < 1$ such that for all $\gamma \in \pi_1(S)$,
$$l(\rho(\gamma)) \leq \lambda\, l(j(\gamma))~.$$
It is straightforward that if a Fuchsian representation $j$ strictly dominates $\rho$, then the length spectrum of $j$ strictly dominates the length spectrum of $\rho$. The converse is highly non-trivial, but happens to be true in the cases we are interested in, where the target space is a smooth complete simply connected Riemannian manifold of sectional curvature bounded above by $-1$ (what we will call for short a \emph{Riemannian $\CAT(-1)$ space}).

\begin{CiteThm}[Gu\'eritaud--Kassel, \cite{GueritaudKassel}]
Let $S$ be a closed oriented surface of negative Euler characteristic, $j$ a Fuchsian representation of $\pi_1(S)$ into $\PSL(2,\R)$, and $\rho$ a representation of $\pi_1(S)$ into the isometry group of a Riemannian $\CAT(-1)$ space. Then $j$ strictly dominates $\rho$ if and only if the length spectrum of $j$ strictly dominates the length spectrum of $\rho$.
\end{CiteThm}

\begin{rmk} \label{rmk:GK}
In \cite{GueritaudKassel}, Gu\'eritaud and Kassel study the minimal Lipschitz constant $\Lip(j,\rho)$ and its relation with the comparison of the length spectra of $j$ and $\rho$ in a broader context. Namely, they consider $j$ a discrete and faithful representation of the fundamental group of a geometrically finite hyperbolic $n$-manifold into $\Isom(\H^n)$. On the other side, they also assume that $\rho$ takes values in the isometry group of $\H^n$. However, as Fran\c cois Gu\'eritaud pointed out to us, their proof would work when replacing $\H^n$ by any Riemannian $\CAT(-1)$ space, since it relies on an equivariant Kirzbraun--Valentine theorem than mainly requires the target space to be ``more negatively curved'' than the base.
\end{rmk}

Given a representation $\rho: \pi_1(S) \to \Isom(M)$, it is natural to ask whether $\rho$ can be dominated by a Fuchsian representation and, if so, what the domain of Fuchsian representations dominating $\rho$ looks like. The first question was answered by Deroin and the author in~\cite{DeroinTholozan} when $(M,g_M)$ is a Riemannian $\CAT(-1)$ space. This applies for instance when $M$ is the symmetric space of a simple Lie group of real rank $1$ (with a suitable normalization of the metric), and in particular for representations in $\PSL(2,\R)$. In that case, it was obtained independently and with other methods by Gu\'eritaud--Kassel--Wolff \cite{GKW}.

\begin{CiteThm}[Deroin--Tholozan]\label{t:Bertrand}
Let $S$ be a surface of negative Euler characteristic, $(M,g_M)$ a Riemannian $\CAT(-1)$ space and $\rho$ a representation of $\pi_1(S)$ into $\Isom(M)$. Then either $\rho$ is Fuchsian in restriction to some stable totally geodesic $2$-plane of curvature $-1$ embedded in $M$, or there exists a Fuchsian representation that strictly dominates $\rho$.
\end{CiteThm}

\begin{rmk}
Note that, by a simple volume argument (see \cite{Thurston86}, proposition 2.1), a Fuchsian representation cannot be strictly dominated. Thus the theorem is optimal.
\end{rmk}

\subsection{Topology of $\Adm(S)$}
Let $(M,g_M)$ be a Riemannian $\CAT(-1)$ space. Consider a Fuchsian representation $j: \pi_1(S) \to \PSL(2,\R)$ and a representation $\rho : \pi_1(S) \to \Isom(M)$. One easily checks that $\Lip(j,\rho)$ only depends on the class of $j$ under conjugation by $\PSL(2,\R)$ and on the class of $\rho$ under conjugation by $\Isom(M)$. We can thus see $\Lip$ as a function defined on \[\Teich(S) \times \Rep(S,\Isom(M))~,\]
where $\Rep(S,\Isom(M))$ denotes the space of conjugacy classes of representations of $\pi_1(S)$ into $\Isom(M)$, and $\Teich(S)$ the space of classes of representations of Euler class $-\chi(S)$ modulo conjugation, which is well-known to be isomorphic to the \emph{Teichm\"uller space} of $S$ (see section \ref{ss:Teichmuller}).

\begin{prop} \label{p:ContinuityLip}
Let $M$ be a Riemannian $\CAT(-1)$ space. Then the function 
\[\Lip: \Teich(S) \times \Rep(S,\Isom(M)) \to \R_+\]
is continuous.
\end{prop}
This proposition is proved by Gu\'eritaud and Kassel in a slightly different context  \cite{GueritaudKassel}. For completeness, we include a proof in the appendix of this article.\\

Let us denote by $\Dom(\rho)$ the domain of $\Teich(S)$ consisting of representations of Euler class $-\chi(S)$ strictly dominating a given representation $\rho$, an by $\Dom(S, \Isom(M))$ the domain of pairs $(j,\rho) \in \Teich(S) \times \Rep(S, \Isom(M))$ such that $j$ strictly dominates $\rho$. By continuity of the function $\Lip$, $\Dom(\rho)$ and $\Dom(S,\Isom(M))$ are open domains. To prove their theorem, Deroin and the author consider in \cite{DeroinTholozan} a certain map $\PPsi_\rho: \Teich(S) \to \Teich(S)$ and show that its image lies in $\Dom(\rho)$. The map $\PPsi_\rho$ is constructed this way: given a point $X$ in $\Teich(S)$ one can associate to the pair $(X,\rho)$ a holomorphic quadratic differential $\PPhi(X,\rho)$ on $X$, obtained as the Hopf differential of an equivariant harmonic map (see section \ref{ss:HopfDiff}). By a theorem of Sampson--Hitchin--Wolf (see section \ref{ss:Teichmuller}) there is a unique Fuchsian representation $j$ up to conjugacy such that $\PPhi(X,j) = \PPhi(X,\rho)$. We set $\PPsi_\rho(X) = j$. (More details are given in section \ref{ss:TheMapPsi}.) Here, we prove that this map is a homeomorphism from $\Teich(S)$ to $\Dom(\rho)$ and that it varies continuously with $\rho$. We thus obtain the following:

\begin{BigThm} \label{t:MonThm}
Let $S$ be a closed oriented surface of negative Euler characteristic and $(M,g_M)$ a Riemannian $\CAT(-1)$ space. Then the domain $\Dom(S,\Isom(M))$ is homeomorphic to
\[\Teich(S) \times \Rep_{nf}(S,\Isom(M))~,\]
where $\Rep_{nf}(S,\Isom(M))$ denotes the domain of $\Rep(S,\Isom(M))$ consisting of representations that are not Fuchsian in restriction to some totally geodesic plane of curvature $-1$.

What's more, this homeomorphism is fiberwise, meaning that it restricts to a homeomorphism from
$\Teich(S)\times \{\rho\}$ to $\Dom(\rho)\times \{\rho\}$ for any $\rho \in \Rep_{nf}(S,\Isom(M,g_M))$.
\end{BigThm}

As a consequence of Theorem \ref{t:MonThm}, the connected components of $\Dom(S,\Isom(M))$ are in \mbox{1-1}~correspondance with the connected components of $\Rep_{nf}(S, \Isom(M))$. In the case where $M$ is the Poincar\'e half-plane, we have, by Kassel's theorem:
\[\Dom(S,\Isom^+(\H^2, g_P)) = \Adm(S)~.\]
Denoting by $\Rep_k(S,\PSL(2,\R))$ the connected component of representations with Euler class $k$ in $\Rep(S,\PSL(2,\R))$, we obtain:

\begin{BigCoro}
Let $S$ be a closed oriented surface of negative Euler characteristic. Then the space $\Adm(S)$ is homeomorphic to 
$$\Teich(S) \times \bigsqcup_{\chi(S) < k < -\chi(S)} \Rep_k(S,\PSL(2,\R))~.$$
In particular, it has $4g-5$ connected components, classified by the Euler class of the non-Fuchsian representation in each pair.
\end{BigCoro}

\subsection{Further applications}

Note that the map $\PPsi_\rho$ depends non trivially on the choice of a normalization of the metric on $M$. Fix a metric $g_0$ on $M$ of sectional curvature $\leq -1$ and a constant $\alpha \geq 1$. Then the metric $\frac{1}{\alpha^2} g_0$ still has sectional curvature $\leq -1$. To mark the dependence of the function $\Lip$ on the metric on the target, we denote by $\Lip_{g}(j,\rho)$ the minimal Lipschitz constant of a $(j,\rho)$-equivariant map from $(\H^2, g_P)$ to $(M,g)$. Then we clearly have $\Lip_{\frac{1}{\alpha^2} g_0} = \frac{1}{\alpha} \Lip_{g_0}$. Hence, if we apply theorem \ref{t:MonThm} to $(M,\frac{1}{\alpha^2} g_0)$, we obtain a description of the space of pairs $(j, \rho) \in \Teich(S) \times \Rep(\pi_1(S),G)$ such that $\Lip_{g_0}(j,\rho) < \alpha$. For $\alpha > 1$, the curvature of $(M,\frac{1}{\alpha^2} g_0)$ is eveywhere $<-1$. Therefore, there is no totally geodesic plane of curvature $-1$ in $(M, \frac{1}{\alpha^2} g_0)$ and we obtain the following:

\begin{BigCoro} \label{c:Lipschitz<C}
Let $S$ be a closed oriented surface of negative Euler characteristic, $(M,g_M)$ a Riemannian $\CAT(-1)$ space and $\alpha$ a constant bigger than $1$. Then the domain
\[\left\{ (j,\rho) \in \Teich(S) \times \Rep(S,\Isom(M)) \mid \Lip_{g_M}(j,\rho) < \alpha \right \} \]
is fiberwise homeomorphic to
\[\Teich(S) \times \Rep(S,\Isom(M))~.\]
\end{BigCoro}

When applied to $M = (\H^2,g_P)$, we obtain a description of left open balls for Thurston's \emph{asymmetric distance} on $\Teich(S)$. Recall that this distance, introduced in \cite{Thurston86}, is defined by
\[d_{Th}(j,j') = \log \left(\Lip_{g_P}(j,j') \right)\] for $j$ and $j'$ any two representations of Euler class $-\chi(S)$. The function $d_{Th}$ is continuous, positive whenever $j$ and $j'$ are distinct, and satisfies the triangular inequality. However, it is not symmetric.

Fix a point $j_0$ in $\Teich(S)$ and a constant $C > 0$. Then, using the convexity of length functions on $\Teich(S)$ \cite{Wolpert87}, one can see that the domain
$$\{j \in \Teich(S) \mid d_{Th}(j_0, j) <C\}$$
is an open convex domain of $\Teich(S)$ for the Weil--Petersson metric (cf.\ section \ref{ss:Teichmuller}). In particular, it is homeomorphic to a ball of dimension $-3\chi(S)$. Since the distance is asymmetric, it is not clear whether the same holds for
$$\{j \in \Teich(S) \mid d_{Th}(j,j_0) <C\}~.$$
However, since 
\[d_{Th}(j,j_0) < C \Leftrightarrow \Lip(j,j_0) < e^C~,\]
we can apply corollary \ref{c:Lipschitz<C} and obtain:

\begin{BigCoro}
Let $S$ be a closed oriented surface of negative Euler characteristic and $j_0$ a point in $\Teich(S)$. Then for any $C > 0$, the domain
\[ \left \{ j \in \Teich(S) \mid d_{Th}(j,j_0) < C \right\} \]
is homeomorphic to $\Teich(S)$.
In other words, left open balls for Thurston's asymetric distance on $\Teich(S)$ are contractible.
\end{BigCoro}

\subsection{Structure of the article and strategy of the proof}
The article is organized as follows. In the next section we recall some fundamental results about harmonic maps from a surface. In particular, we recall the Corlette--Labourie theorem of existence of equivariant harmonic maps and the Sampson--Hitchin--Wolf parametrization of the Teichm\"uller space by means of quadratic differentials. In the second section, we start by using those theorems to construct the map $\PPsi_\rho$ studied in \cite{DeroinTholozan}, and we prove that $\PPsi_\rho$ is a homeomorphism from $\Teich(S)$ to $\Dom(\rho)$. To do so, we make explicit the inverse of the map $\PPsi_\rho$. It will appear that reverse images of a point $j$ in $\Teich(S)$ by $\PPsi_\rho$ are exactly critical points of a certain functional $\F_{j,\rho}$. We will prove that when $j$ is in $\Dom(\rho)$, the functional $\F_{j,\rho}$ is proper and admits a unique critical point which is a global minimum. Hence the map $\PPsi_\rho$ is bijective. What's more, the functionals $\F_{j,\rho}$ vary continuously with $(j,\rho)$, and so does their unique minimum. This will prove the continuity of $\PPsi_\rho^{-1}$ and its continuous dependance in $\rho$.\\

We wish to thank Richard Wentworth and Mike Wolf for their interest in this work and their help in understanding some regularity question concerning harmonic maps. We also thank Olivier Guichard and Fanny Kassel for useful remarks on a previous version of this paper.

\section{Representations of surface groups, harmonic maps, and Teichm\"uller space} \label{s:harmonic}

In this section, we introduce briefly the tools from the theory of harmonic maps that we will need later. We refer to \cite{DaskalopoulosWentworth} for a more thorough survey.\\

\subsection{Existence theorems}

Recall that the \emph{energy density} of a non-negative symmetric \mbox{$2$-form $h$} on a Riemannian manifold $(X,g)$ is the function on $X$ defined as
$$e_g(h) = \frac{1}{2}\Tr(A(h))~,$$
where $A(h)$ is the unique field of endomorphisms of the tangent bundle such that for all $x\in X$ and all $u,v \in T_x X$, 
$$h_x(u,v) = g_x(u, A(h)_x v)~.$$
If $f$ is a map between two Riemannian manifolds $(X, g_X)$ and $(Y, g_Y)$ then its energy density is the energy density on $X$ of the pull-back metric:
$$e_{g_X}(f) = e_{g_X} (f^*g_Y)~.$$
This energy density can be integrated against the volume form $\Vol_{g_X}$ associated to $g_X$, giving the \emph{total energy} of $f$:
$$E_{g_X}(f) = \int_X e_{g_X}(f) \Vol_{g_X}~.$$

The map $f$ is called \emph{harmonic} if it is, in some sense, a critical point of  the total energy. For instance, harmonic maps from $\R$ to a Riemannian manifold are geodesics and harmonic maps from a Riemannian manifold to $\R$ are harmonic functions. In general, a map $f$ is harmonic if it satisfies a certain partial differential equation that can be expressed as the vanishing of the \emph{tension field}. We give and use the precise definition of harmonicity in \cite{DeroinTholozan}. Here we will be satisfied with existence results and some fundamental properties of those maps.

The first existence result is due to Eells and Sampson.
\begin{CiteThm}[Eells--Sampson, \cite{EellsSampson}]
Let $f: (X, g_X) \to (Y,g_Y)$ be a continuous map between two closed Riemannian manifolds. Assume $(Y, g_Y)$ has non-positive sectional curvature. Then there exists a harmonic map $f':  (X, g_X) \to (Y,g_Y)$ homotopic to $f$, which minimizes the energy among all maps homotopic to $f$. Moreover, if the sectional curvature of $Y$ is negative, then this map is unique, unless the image of $f'$ is included in a geodesic of $Y$.
\end{CiteThm}

Eells and Sampson's paper contains a thorough study of the analytic aspects of harmonicity, which allows to extend their existence result in several cases. The one we will be interested~in is an equivariant version. Consider $(X,g_X)$ a closed Riemannian manifold, $(\tilde{X}, \tilde{g}_X)$ its universal cover, $(M,g_M)$ another Riemannian manifold and $\rho$ a representation of $\pi_1(X)$ into $\Isom(M)$. A map $f: \tilde{X} \to M$ is called \emph{$\rho$-equivariant} if for all $x\in \tilde{X}$ and all $\gamma \in \pi_1(X)$, we have
$$f(\gamma \cdot x) = \rho(\gamma) \cdot f(x).$$
Given such a map, the symmetric $2$-form $f^*g_M$ on $\tilde{X}$ is invariant under the action of $\pi_1(X)$ and thus induces a symmetric $2$-form on $X$. We define the energy density and the total energy of $f$ as the energy density and the total energy of this symmetric $2$-form.

The equivariant version of Eells--Sampson's theorem is due to Corlette \cite{Corlette88} in the specific case where $M$ is a symmetric space of non-compact type, and to Labourie \cite{Labourie91} in the more general case of a complete simply connected Riemannian manifold of non-positive curvature. We state it in the particular case where $(M,g_M)$ is a Riemannian $\CAT(-1)$ space. This implies that $(M,g_M)$ is \emph{Gromov hyperbolic}. Recall that, for such a space, there is a good notion of \emph{boundary at infinity}, such that every isometry of $M$ extends to a homeomorphism of the boudary  (see section \ref{as:CAT(-1)} of the appendix).

\begin{CiteThm}[Corlette, Labourie]
Let $(X,g_X)$ be a closed Riemannian manifold, $\tilde{X}$ its universal cover, and $(M,g_M)$ a Riemannian $\CAT(-1)$-space. Let $\rho$ be a representation of $\pi_1(X)$ into $\Isom(M)$. Assume that $\rho$ does not fix a point in $\partial_\infty M$. Then there exists a unique harmonic map from $(\tilde{X}, g_{\tilde{X}})$ to $(M,g_M)$ that is $\rho$-equivariant. This map minimizes the energy among all such equivariant maps.
\end{CiteThm}

\subsubsection{Dealing with representations fixing a point at infinity} \label{ss:ParabolicRep}

In the sequel, representations fixing a point in $\partial_\infty M$ will be called \emph{parabolic representations}. When $M$ is the symmetric space of some rank $1$ Lie group $G$, those are exactly the representations taking values into some parabolic subgroup of $G$. For some of those representations, the theorem of Corlette and Labourie does not apply, and we must say a few words in order to be able to include them in the proof later.

The main problem with a parabolic representation $\rho$ is that a sequence of equi-Lipschitz $\rho$-equivariant maps from $\tilde{X}$ to $M$ may not converge. In the appendix, we deal with this issue in order to prove the continuity of the function $\Lip(j, \, \cdot \,)$ at a parabolic representation. With the same ideas and a few notions of harmonic analysis, one can prove the following fact:

\begin{prop}
Let $\rho : \pi_1(X) \to \Isom(M))$ be a parabolic representation. Let $m_\rho$ be the morphism from $\pi_1(X)$ to $\R$ such that 
\[ | m_\rho(\gamma) | = l(\rho(\gamma)) \]
for all $\gamma \in \pi_1(X)$ (see lemma \ref{l:CaracterisationsRepresentationsReductibles} in the appendix).

Let $f_n$ be a sequence of $\rho$-equivariant maps from $\tilde{X}$ to $M$ whose energy converges to the infimum of the energies of all such equivariant maps. Then the symmetric $2$-form $f_n^*g_M$ converges to 
\[f^* \d x^2~,\]
where $f$ is a $m_\rho$-equivariant harmonic function from $\tilde{X}$ to $\R$ and $\d x^2$ is the canonical metric on $\R$. (The function $f$ is obtained by integrating the unique harmonic $1$-form on $X$ whose periods are given by $m_\rho$, and is unique up to a translation.)
\end{prop}

In some sense, this $m_\rho$-equivariant function $f$ is the natural extension of the notion of $\rho$-equivariant harmonic maps to parabolic representations. In particular, when $\rho$ fixes two points in $\partial_\infty M$, it stabilizes a unique geodesic $c: \R \to M$; in that case, $\rho$-equivariant harmonic maps exist and have the form 
\[ c \circ f~.\]

\subsection{Harmonic maps from a surface, Hopf differential, and Teichm\"uller space}

From now on we will restrict to harmonic maps from a Riemann surface. In that case, the energy of a map only depends on the conformal class of the Riemannian metric on the base. For the same reason, harmonicity is invariant under a conformal change of the metric.

\subsubsection{Hopf differential} \label{ss:HopfDiff}

Let $S$ be an oriented surface equipped with a Riemannian metric $g$. Let $(M,g_M)$ be a Riemannian manifold, and $f: S \to M$ a smooth map. The conformal class of $g$ induces a complex structure on $S$. The symmetric $2$-form $f^*g_M$ can thus be uniquely decomposed into a $(1,1)$ part, a $(2,0)$ part and a $(0,2)$ part. One can check that the $(1,1)$~part is $e_g(f) g$, and we thus have
$$f^*g_M = e_g(f) g  + \Phi_f + \bar{\Phi}_f~,$$
where $\Phi_f$ is a \emph{quadratic differential} (i.e.\ a section of the square of the canonical bundle of $(S,[g])$ called the \emph{Hopf differential} of $f$.

The following proposition is classical.

\begin{prop}
If the map $f$ is harmonic, then its Hopf differential is holomorphic. The converse is true if $M$ is also a surface.
\end{prop}

\subsubsection{The Teichm\"uller space} \label{ss:Teichmuller}
Let $S$ be a closed oriented surface of negative Euler characteristic. Recall that the Teichm\"uller space of $S$, denoted $\Teich(S)$, is the space of complex structures on $S$ compatible with the orientation, where two complex structures are identified if there is a biholomorphism between them isotopic to the identity.

By Poincar\'e--Koebe's uniformization theorem, any complex structure on $S$ admits a unique conformal Riemannian metric of constant curvature $-1$. Therefore, $\Teich(S)$ can also be seen as the space of hyperbolic metrics on $S$ up to isotopy. Lastly, any hyperbolic metric $g$ on $S$ induces a Fuchsian representation $j: \pi_1(S) \to \PSL(2,\R)$, such that $(S,g)$ identifies isometrically with $j(\pi_1(S)) \backslash \H^2$ ($j$ is called a \emph{holonomy representation} of the metric $g$). The representation $j$ is unique up to conjugation, has Euler class $-\chi(S)$, and $\Teich(S)$ thus identifies canonically with the space of  representations from $\pi_1(S)$ to $\PSL(2,\R)$ of Euler class $-\chi(S)$ modulo conjugation. Throughout the paper, a point in $\Teich(S)$ is denoted alternately by the letter $X$ when we think of it as the surface $S$ equipped with a complex structure, or by the letter $j$ when we think of it as a Fuchsian representation.

It is well known that the Teichm\"uller space is a manifold diffeomorphic $\R^{-3 \chi(S)}$ and that it carries a complex structure. Consider two points $X_1$ and $X_2$ in $\Teich(S)$ corresponding to two hyperbolic metrics $g_1$ and $g_2$ on $S$. Then, by Eells--Sampson's theorem, there is a unique harmonic map {$f_{g_1, g_2}: (S,g_1) \to (S,g_2)$} homotopic to the identity map. Schoen--Yau's theorem \cite{SchoenYau} (also proved by Sampson \cite{Sampson78} in that specific case) states that this map is a diffeomorphism. 

Of course, the map $f_{g_1,g_2}$ depends on the choice of $g_1$ and $g_2$ up to isotopy. Actually, fixing $g_1$ or $g_2$, one can choose the other metric so that the identity map itself is harmonic (by replacing $g_2$ by $f_{g_1,g_2}^*g_2$ or $g_1$ by $f_{g_1,g_2}^*g_1$). On the other hand, the total energy of $f_{g_1,g_2}$ is invariant under changing one of the metrics by isotopy, and thus gives a well defined functional
$$\E: \Teich(S) \times \Teich(S) \to \R_+~.$$

Besides, the Hopf differential of $f_{g_1,g_2}$ is invariant under isotopic changes of $g_2$, and if $h$ is a diffeomorphism of $S$, one has
$$\Phi_{f_{h^*g_1, g_2}} = h^* \Phi_{f_{g_1, g_2}}~.$$
The Hopf differential thus induces a well defined map 
$$\PPhi: \Teich(S) \times \Teich(S) \to \QD\Teich(S)~,$$
where $\QD\Teich(S)$ denotes the complex bundle of holomorphic quadratic differentials on $\Teich(S)$, that is, $\QD_X \Teich(S)$ is the space of holomorphic quadratic differentials on $X$. 

It is known since the work of Teichm\"uller that this bundle identifies with the tangent bundle to $\Teich(S)$. Sampson \cite{Sampson78} proved that the map $\PPhi$ is an injective immersion. In particular, if one fixes a point $X_0 \in \Teich(S)$, the derivative at $X_0$ of the map $\PPhi(X_0, \, \cdot \,)$ provides a new identification of the tangent space $\T_{X_0}\Teich(S)$ with $\QD_{X_0}\Teich(S)$.

Finally, Wolf \cite{Wolf89} proved that the map $\PPhi$ is a global homeomorphism. This was obtained independently by Hitchin \cite{Hitchin87}, as the first construction of a section to the \emph{Hitchin fibration}.

\begin{CiteThm}[Sampson, Hitchin, Wolf]
The map $\PPhi: \Teich(S) \times \Teich(S) \to \QD \Teich(S)$ is a homeomorphism.
\end{CiteThm}

\subsubsection{The Weil--Petersson metric on $\Teich(S)$}

Consider $X\in \Teich(S)$, and choose $g$ a hyperbolic metric representing $X$. Then $g$ naturally induces a Hermitian metric on the line bundle $K_X^2$. In a local conformal coordinate $z$, write
$$g = \alpha\, \d z \d \bar{z}$$
for some positive function $\alpha$. Then a section $\Phi$ of $K_X^2$ is locally of the form $\phi\, \d z^2$, and the Hermitian metric on $K_X^2$ induced by $g$ is given by
$$\norm{\Phi}_g^2 = \frac{|\phi|^2}{\alpha^2}~.$$
(One can check that this function is independent of the choice of a local conformal coordinate.)

This Hermitian metric allows to define a Hermitian norm on the complex vector space $\QD_X \Teich(S)$ of holomorphic quadratic differentials on $X$, given by
$$\norm{\Phi}_{WP}^2 = \int_S \norm{\Phi}_g^2 \Vol_g~.$$
Since $\QD_X \Teich(S)$ naturally identifies with $\T_X \Teich(S)$, this defines a Riemannian metric on $\Teich(S)$, called the \emph{Weil--Petersson metric}. Many interesting properties of the Weil-Petersson metric have been proved by Ahlfors \cite{Ahlfors61} and Wolpert \cite{Wolpert87}, among others. Here we will be satisfied with this simple definition.

\subsection{Equivariant harmonic maps and functionals on $\Teich(S)$}  \label{ss:EnergyFunctional}

The theorem of Corlette and Labourie allows to extend the maps $\E$ and $\PPhi$ to $\Teich(S) \times \Rep(S, \Isom(M))$. Given a point in $X_0 \in \Teich(S)$, represented by a hyperbolic metric $g_0$, and a point $\rho$ in $\Rep(\pi_1(S), \Isom(M))$, one can consider $f_{g_0,\rho}$ the unique $\rho$-equivariant harmonic map from $(\tilde{S}, \tilde{g}_0)$ to $(M,g_M)$. (If $\rho$ is parabolic, $f_{g_0,\rho}$ will denote instead a $m_\rho$-equivariant harmonic function, see section \ref{ss:ParabolicRep}.) The symmetric $2$-form $f_{g_0,\rho}^*g_M$ (or $f_{g_0,\rho}^*\d x^2$ when $\rho$ is parabolic) on $\tilde{S}$ is $\pi_1(S)$-invariant, hence the energy density and the Hopf differential of $f_{g_0,\rho}$ are also $\pi_1(S)$-invariant. They thus descend on $S$. We denote them $e_{g_0}(f_{g_0,\rho})$ and $\Phi_{f_{g_0,\rho}}$.

The energy density and the Hopf differential of $f_{g_0,\rho}$ only depend on the conjugacy class of the representation $\rho$. We thus obtain two well-defined maps
$$\function{\E}{\Teich(S) \times \Rep(S, \Isom(M))}{\R_+}{(X_0,\rho)}{\int_S e_{g_0}(f_{g_0,\rho}) \Vol_0}$$
and
$$\function{\PPhi}{\Teich(S) \times \Rep(S, \Isom(M))}{\QD \Teich(S)}{(X_0,\rho)}{\Phi_{f_{g_0,\rho}}~.}$$

\begin{rmk}
If $M = \H^2$, we have $\Isom^+(M) \simeq \PSL(2,\R)$. Consider two points $X, X_1 \in \Teich(S)$, and let $j_1$ be a holonomy representation of a  hyperbolic metric $g_1$ corresponding to $X_1$. Then a harmonic map from $X$ to $X_1$ isotopic to the identity lifts to a $j_1$-equivariant harmonic map from $\tilde{X}$ to $\H^2$.
We thus have
$$\E(X, j_1) = \E(X,X_1)$$
and
$$\PPhi(X, j_1) = \PPhi(X, X_1)~.$$
In other words, the new definition of $\E$ and $\PPhi$ extends the one given in the previous paragraph, via the identification of $\Teich(S)$ with the component of maximal Euler class in $\Rep(S, \PSL(2,\R))$.
\end{rmk}

\begin{prop} \label{p:ContinuityEPhi}
The function $\E(X,\rho)$ and the map $\PPhi(X,\rho)$ are continuous with respect to $X$ and $\rho$.
\end{prop}

\begin{proof}
This is a classical consequence of the ellipticity of the equation defining harmonic maps. If $(j_n, \rho_n)$ converges to $(j,\rho)$, the derivatives of a $(j_n,\rho_n)$-equivariant harmonic map $f_n$ can be uniformly controlled by its total energy. If $\rho$ is not parabolic, one can deduce that the sequence $f_n$ converges in $C^1$ topology to a $(j,\rho)$-equivariant harmonic map $f$ (see section \ref{ass:LowerContinuity} in the appendix). The proposition easily follows.

The case where $\rho$ is parabolic can also be solved using the same ideas as those in the proof of the continuity of the minimal Lipschitz constant.
\end{proof}

We will use the following important results.

\begin{prop}[see \cite{Wentworth07}] \label{t:GradEnergy}
Let $(M,g_M)$ be a Riemannian $\CAT(-1)$ space and $\rho$ a representation of $\pi_1(S)$ into $\Isom(M)$. Then the functional
$$\function{\E(\, \cdot \, , \rho)}{\Teich(S)}{\R_+}{X}{\E(X,\rho)}$$
is $C^1$ and its differential at a point $X_0 \in \Teich(S)$ is given by
$$\Grad_{WP} \E(\, \cdot \, , \rho)(X_0) = - \PPhi(X_0, \rho)$$
(where $\Grad_{WP}$ denotes the gradient with respect to the Weil--Petersson metric).
\end{prop}

\begin{thm}[Tromba, \cite{Tromba92}] \label{t:PropernessEnergy}
Consider a point $j$ in $\Teich(S)$. Then the function
$$\function{\E(\, \cdot \, , j)}{\Teich(S)}{\R_+}{X}{\E(X,j)}$$
is proper.
\end{thm}

\section{The homeomorphism from $\Teich(S)$ to $D(\rho)$} \label{s:proof}

We are now in possession of all the tools required to define the map $\PPsi_\rho$ introduced in \cite{DeroinTholozan} and to prove that it is a homeomorphism from $\Teich(S)$ to $\Dom(\rho)$.

\subsection{Construction of the map $\PPsi_\rho$} \label{ss:TheMapPsi}

Let $S$ be a closed oriented surface of negative Euler characteristic and $\rho$ a representation of $\pi_1(S)$ into the isometry group of a Riemannian $\CAT(-1)$ space $(M,g_M)$. Let $X_1$ be a point in $\Teich(S)$. Then $\PPhi(X_1,\rho)$ is a holomorphic quadratic differential on $X_1$, and the theorem of Sampson--Hitchin--Wolf asserts that there is a unique point $j_2$ in $\Teich(S)$ such that $\PPhi(X_1, j_2) = \PPhi(X_1, \rho)$. Setting
$$\PPsi_\rho(X_1) = j_2~,$$
we obtain a well-defined map
$$\PPsi_\rho: \Teich(S)\to \Teich(S)~.$$
This map only depends on the class of $\rho$ under conjugation by $\Isom(M)$. We can thus define a map
$$\function{\PPsi}{\Teich(S) \times \Rep(S,\Isom(M))}{\Teich(S) \times \Rep(S, \Isom(M))}{(X,\rho)}{( \PPsi_\rho(X), \rho)~.}$$

In \cite{DeroinTholozan}, the following is proved:

\begin{CiteThm}[Deroin--Tholozan]
Either $\rho$ preserves a totally geodesic $2$-plane of curvature $-1$ in restriction to which it is Fuchsian, or the image of the map $\PPsi_\rho$ lies in $\Dom(\rho)$. In particular, $\Dom(\rho)$ is non empty.
\end{CiteThm}

\subsection{Surjectivity of the map $\PPsi_\rho$}

We first prove the following 
\begin{prop}
The map $\PPsi_\rho: \Teich(S) \to \Dom(\rho)$ is surjective.
\end{prop}

\begin{proof}
Fix $j_0\in \Teich(S)$ and $\rho \in \Rep(S, \Isom(M))$. Let us introduce the functional $$\function{\F_{j_0,\rho}}{ \Teich(S)}{\R}{X}{\E(X,j_0) - \E(X,\rho)~.}$$
By proposition \ref{t:GradEnergy}, the map $\F_{j_0,\rho}$ is $C^1$ and its gradient with respect to the Weil--Petersson metric is given by
$$\Grad_X(\F_{j_0,\rho}) = - \PPhi(X, j_0) + \PPhi(X,\rho)~.$$
Hence $X_1$ is a critical point of $\F_{j_0,\rho}$ if and only if $\PPhi(X_1, j_0) = \PPhi(X_1,\rho)$, which means precisely that 
$$\PPsi_\rho(X_1) = j_0~.$$
Proving that $j_0$ is in the image of $\PPsi_\rho$ is thus equivalent to proving that the map $\F_{j_0,\rho}$ admits a critical point. This will be a consequence of the following lemma:

\begin{lem}\label{l:properness}
For $j_0 \in \Teich(S)$ and $\rho \in \Rep(S, \Isom(M))$, we have the following inequality: 
$$\F_{j_0,\rho} \geq \left(1-\Lip(j_0,\rho)\right) \E(\, \cdot \,, j_0)~.$$
\end{lem}

\begin{proof}
Let $f$ be a $(j_0, \rho)$-equivariant Lipschitz map from $\H^2$ to $M$ with Lipschitz constant $\Lip(j_0,\rho) + \epsilon$.
Let $X$ be a point in $\Teich(S)$ represented by some hyperbolic metric $g$, and let $h$ be the unique $j_0$-equivariant map from $\tilde{X}$ to $\H^2$. We thus have $E_g(h) = \E(X,j_0)$. Besides, the map $f\circ h$ is a $\rho$-equivariant map from $\tilde{X}$ to $M$, and we thus have
$$\E(X,\rho) \leq E_g(f \circ h)~.$$
Since $f$ is $(\Lip(j_0,\rho) + \epsilon)$-Lipschitz, we have
$$\E(X,\rho) \leq E_g(f \circ h) \leq (\Lip(j_0,\rho) + \epsilon) E_g(h) = (\Lip(j_0,\rho) + \epsilon) \E(X,j_0)~,$$
from which we get
$$\F_{j_0,\rho}(X) = \E(X,j_0) - \E(X,\rho) \geq (1-\Lip(j_0,\rho) - \epsilon) \E(X,j_0).$$
This being true for any $\epsilon > 0$, we obtain the required inequality.
\end{proof}

Now, by Theorem \ref{t:PropernessEnergy}, the map $X\to \E(X,j_0)$ is proper. Therefore, if $j_0$ is in $\Dom(\rho)$, we have $(1- \Lip(j_0,\rho)) > 0$ and the function $\F_{j_0,\rho}$ is also proper. Hence $\F_{j_0,\rho}$ admits a minimum. This minimum is a critical point, and thus a pre-image of $j_0$ by the map $\PPsi_\rho$. We obtain that $\PPsi_\rho: \Teich(S) \to \Dom(\rho)$ is surjective.
\end{proof}

We proved that if $j_0$ is in $\Dom(\rho)$, the functional $\F_{j_0,\rho}$ is proper. Note that, conversely, if $\F_{j_0,\rho}$ is proper, then it admits a critical point. Hence $j_0$ is in the image of $\PPsi_\rho$, which implies that $j_0$ lies in $\Dom(\rho)$ by \cite{DeroinTholozan}. We thus obtain the following corollary that might be interesting in its own way:

\begin{coro}
The following are equivalent:
\begin{itemize}
\item[$(i)$] the map $\F_{j_0, \rho}$ is proper,
\item[$(ii)$] the map $\F_{j_0,\rho}$ admits a critical point,
\item[$(iii)$] the representation $j_0$ strictly dominates $\rho$.
\end{itemize}
\end{coro}

\subsection{Injectivity of the map $\PPsi_\rho$}

To prove injectivity, we need to prove that when $j_0$ is in $\Dom(\rho)$, the critical point of $\F_{j_0,\rho}$ is unique. To do so, we prove that any critical point of $\F_{j_0,\rho}$ is a strict minimum of $\F_{j_0,\rho}$.

Let $X_1$ be a critical point of $\F_{j_0,\rho}$, and $X_2$ another point in $\Teich(S)$. Choose a hyperbolic metric $g_1$ on $S$ representing $X_1$. Let $g_0$ be the hyperbolic metric of holonomy $j_0$ such that $\Id: (S,g_1) \to (S,g_0)$ is harmonic, and $g_2$ the hyperbolic metric representing $X_2$ such that $\Id: (S,g_2)\to (S,g_0)$ is harmonic. Let $f: (\tilde{S}, \tilde{g}_1)\to (M,g_M)$ be a $\rho$-equivariant harmonic map. We have the following decompositions:
$$g_0 = e_{g_1}(g_0) g_1 + \Phi + \bar{\Phi},$$
$$f^*g_M = e_{g_1}(f) g_1 + \Phi + \bar{\Phi},$$
$$g_2 = e_{g_1}(g_2) g_1 + \Psi + \bar{\Psi},$$
where $\Phi$ and $\Psi$ are quadratic differentials on $S$, with $\Phi$ holomorphic with respect to the complex structure induced by $g_1$. Note that the same $\Phi$ appears in the decomposition of $g_0$ and $f^*g_M$ because $X_1$ is a critical point of $\F_{j_0,\rho}$, and thus $\PPhi(X_1,j_0) = \PPhi(X_1,\rho)$.

\begin{lem}\label{l:compare-energie}
We have the following identity:
\[ E_{g_2}(g_0) - E_{g_2}(f^*g_M) = \int_S \frac{1}{\sqrt{1- \frac{4\norm{\Psi}_{g_1}^2}{e_{g_1}(g_2)^2}}} \left( e_{g_1}(g_0) - e_{g_1}(f^*g_M)\right) \Vol_{g_1}~. \]
\end{lem}

\begin{proof}[Proof of lemma \ref{l:compare-energie}]
This is a rather basic computation that we will carry out in local coordinates. Let $z = x+ i y$ be a local complex coordinate with respect to which $g_1$ is conformal. We denote by $\scal{\, \cdot \,}{\, \cdot \,}$ the standard scalar product on $\C$

Any symmetric $2$-form on $\C$ can be written under the form $\scal{\, \cdot \,}{G \, \cdot \,}$, where $G$ is a field of symmetric endomorphisms of $\R^2$ depending on the coordinates $(x,y)$. We will represent such an endomorphism by its matrix in the canonical frame $(\frac{\partial}{\partial x}, \frac{\partial}{\partial y})$.
In local coordinates, we can thus write
$$g_0 =\scal{\, \cdot \,}{G_0 \, \cdot \,}~,$$
$$g_1= \scal{\, \cdot \,}{G_1 \, \cdot \,}~,$$
$$g_2= \scal{\, \cdot \,}{G_2 \, \cdot \,}~,$$
$$f^*g_M= \scal{\, \cdot \,}{G_f \, \cdot \,}~.$$
Now, since $g_1$ is conformal with respect to the coordinate $z$, we have $g_1 = \alpha \scal{\, \cdot \,}{\, \cdot \,}$ for some positive function $\alpha$, and we can write
$$\Phi = \phi \,\d z^2~,$$
$$\Psi = \psi \,\d z^2~,$$
for some complex valued functions $\phi$ and $\psi$. (Since $\Phi$ is holomorphic, $\phi$ must be holomorphic, but we won't need it for our computation.)

We can now express $G_0$, $G_1$, $G_2$, $G_f$, $\Vol_{g_1}$ and $\Vol_{g_2}$ in terms of $\alpha$, $e_{g_1}(g_0)$, $e_{g_1}(g_2)$, $e_{g_1}(f)$, $\phi$ and $\psi$. We easily check that
$$G_1 = \left(\begin{matrix}
\alpha & 0 \\
0 & \alpha 
\end{matrix}\right)~,$$

$$G_0 = \left(\begin{matrix}
\alpha e_{g_1}(g_0) + 2\Re(\phi) & -2 \Im(\phi) \\
-2\Im(\phi) & \alpha  e_{g_1}(g_0) - 2 \Re(\phi)
\end{matrix}\right)~,$$

$$G_f = \left(\begin{matrix}
\alpha e_{g_1}(g_f) + 2\Re(\phi) & -2 \Im(\phi) \\
-2\Im(\phi) & \alpha  e_{g_1}(g_f) - 2 \Re(\phi)
\end{matrix}\right)~,$$

$$G_2 = \left(\begin{matrix}
\alpha e_{g_1}(g_2) + 2\Re(\psi) & -2 \Im(\psi) \\
-2\Im(\psi) & \alpha  e_{g_1}(g_2) - 2 \Re(\psi)
\end{matrix}\right)~,$$

$$\Vol_{g_1} = \alpha \d z \d \bar{z}~,$$

$$\Vol_{g_2} = \sqrt{\det G_2} \d z \d \bar{z} = \sqrt{\alpha^2 e_{g_1}(g_2)^2 - 4 |\psi|^2} \d z \d \bar{z}~.$$
We now want to express $e_{g_2}(g_0)$ and $e_{g_2}(f^*g_M)$. To do so, note that we can write
\begin{eqnarray*}
g_0(\, \cdot \,, \, \cdot \,) &= & \scal{\, \cdot \,}{ G_0 \, \cdot \,}\\
\ & = & \scal{\, \cdot \,}{ G_2(G_2^{-1}G_0) \, \cdot \,}\\
\ & = & g_2(\, \cdot \,, G_2^{-1} G_0 \, \cdot \,)~.
\end{eqnarray*}
By definition of the energy density, we thus obtain that
\begin{eqnarray*}
e_{g_2}(g_0) &=& \frac{1}{2}\Tr(G_2^{-1}G_0)\\
\ & = & \frac{1}{2} \Tr \left[
                     \frac{1}{\det G_2} \left( \begin{matrix}
                                                   											\alpha e_{g_1}(g_2) - 2\Re(\psi) & 2 \Im(\psi) \\
																								2\im(\psi) & \alpha  e_{g_1}(g_2) + 2 \Re(\psi)
																								\end{matrix}\right)	\right.\\
																								\ & \ & \times	\left.						
																								\left(\begin{matrix}
																								\alpha e_{g_1}(g_0) + 2\Re(\phi) & -2 \Im(\phi) \\
																								-2\im(\phi) & \alpha  e_{g_1}(g_0) - 2 \Re(\phi)
																								\end{matrix}\right) \right] \\
\ & = & \frac{1}{\alpha^2 e_{g_1}(g_2)^2 - 4 |\psi|^2} \left( \alpha^2 e_{g_1}(g_2) e_{g_1}(g_0) - ( \phi \bar{\psi} + \bar{\phi} \psi) \right) .
\end{eqnarray*}
Similarly, we get that 
$$e_{g_2}(f^*g_M) = \frac{1}{\alpha^2 e_{g_1}(g_2)^2 - 4 |\psi|^2} \left( \alpha^2 e_{g_1}(g_2) e_{g_1}(f) - ( \phi \bar{\psi} + \bar{\phi} \psi) \right)~.$$
When computing the difference, the terms $\phi \bar{\psi} + \bar{\phi} \psi$ simplify. (Here we use the fact that $f^*g_M$ and $g_0$ have the same $(2,0)$-part.) We eventually obtain

\begin{eqnarray*}
\left( e_{g_2}(g_0) - e_{g_2}(f^*g_M) \right) \Vol_{g_2} & = & \frac{\alpha^2 e_{g_1}(g_2) (e_{g_2}(g_0) - e_{g_2}(f^*g_M))}{\sqrt{\alpha^2 e_{g_1}(g_2)^2 - 4 |\psi|^2}}\, \d z \d \bar{z} \\
\ & = & \frac{e_{g_2}(g_0) - e_{g_2}(f^*g_M)}{\sqrt{1 - 4 \frac{\norm{\Psi}_{g_1}^2}{e_{g_1}(g_2)^2}}} \Vol_{g_1}~.
\end{eqnarray*}
Now, the parameters of the last expression are well defined functions on $S$ and the identity is true in any local chart. It is thus true everywhere on $S$ and, when integrating, we obtain lemma~\ref{l:compare-energie}.
\end{proof}

From lemma \ref{l:compare-energie}, we obtain that
$$E_{g_2}(g_0) - E_{g_2}(f^*g_M) \geq \int_S (e_{g_1}(g_0) - e_{g_1}(f^*g_M) ) \Vol_{g_1} = E_{g_1}(g_0) - E_{g_1}(f^*g_M) = \F_{j_0, \rho} (X_1)~,$$
with equality if and only if $\norm{\Psi}_{g_1} \equiv 0$, that is, if $g_1$ is conformal to $g_2$.

On the other side, we have
$E_{g_2}(g_0) = \E(X_2, j_0)$ (since, by hypothesis, the identity map from $(S,g_2)$ to $(S,g_0)$ is harmonic) and $E_{g_2}(f^*g_M) \geq \E(X_2, \rho)$, from which we deduce that 
$$E_{g_2}(g_0) - E_{g_2}(f^*g_M) \leq \F_{j_0, \rho}(X_2)~.$$
Combining the two inequalities, we obtain that
$$\F_{j_0, \rho}(X_2) \geq \F_{j_0, \rho}(X_1)~,$$
with equality if and only if $X_1 = X_2$.

Now, if $X_1$ and $X_2$ are two critical points of $\F_{j_0, \rho}$, then by symmetry we must have $\F_{j_0, \rho}(X_2) = \F_{j_0, \rho}(X_1),$ and therefore $X_2 = X_1$. The functional $\F_{j_0, \rho}$ admits a unique critical point, and $j_0$ admits a unique pre-image by $\PPsi_\rho$. Thus $\PPsi_\rho$ is injective.

\newpage

\subsection{bi-continuity of $\PPsi$} \label{ss:InverseContinuous}

Recall that one has, by definition, 
$$(X, \PPsi_\rho(X)) = \PPhi^{-1}( X, \PPhi(X,\rho))~,$$
where $\PPhi^{-1}$ denotes the inverse of the map $\PPhi : \Teich(S) \times \Teich(S) \to \QD \Teich(S)$, which is a homeomorphism by Sampson--Hitchin--Wolf's theorem. Therefore, by proposition \ref{p:ContinuityEPhi}, the maps $\PPsi_\rho$ and $\PPsi$ are continuous.

Let us now prove that $\PPsi^{-1}$ is continuous. We saw that $\PPsi^{-1}(j,\rho)$ is the unique critical point of a proper function $\F_{j,\rho}$ on $\Teich(S)$ which depends continuously on $j$ and $\rho$. The continuity of $\PPsi^{-1}$ will follow from the fact that the functions $\F_{j,\rho}$ are in some sense \emph{locally uniformly proper}.

\begin{definition}
Let $X$ and $Y$ be two metric spaces, and $\left( F_y\right)_{y\in Y}$ a family of continuous functions from $X$ to $\R$ depending continuously on $y$ for the compact-open topology. We say that the family $\left(F_y \right)_{y\in Y}$ is \emph{uniformly proper} if for any $C\in\R$, there exists a compact subset $K$ of $X$ such that for all $y\in Y$ and all $x \in X \backslash K$, $F_y(x) > C$.

We say that the family $\left(F_y \right)_{y\in Y}$ is \emph{locally uniformly proper} if for any $y_0 \in Y$, there is a neighbourhood $U$ of $y_0$ such that the sub-family $\left(F_y \right)_{y\in U}$ is uniformly proper.
\end{definition}

\begin{prop} \label{p:UniformProp}
Let $X$ and $Y$ be two metric spaces and $\left( F_y\right)_{y\in Y}$ a locally uniformly proper family of continuous functions from $X$ to $\R$ depending continuously on $y$ (for the compact-open topology). Assume that each $F_y$ achieves its minimum at a unique point $x_m(y) \in X$. Then the function 
$$y \mapsto x_m(y)$$ is continuous.
\end{prop}

\begin{proof}
Let us denote by $m(y) = F_{y}( x_m(y))$ the minimum value of $F_y$. Fix $y_0 \in Y$. Let $U$ be a neighbourhood of $y_0$ and $K$ a compact subset of $X$ such that for all $y\in U$ and all $x\in X \backslash K$, we have $$F_y(x) > m(y_0) + 1~.$$ For $\epsilon >0$, define 
$$V_\epsilon = \{x \in X \mid F_{y_0}(x) < m(y_0) + \epsilon \}~.$$ Since $F_{y_0}$ is proper and achieves its minimum at a single point $x_m(y_0)$, the family $\left(V_\epsilon \right)_{\epsilon > 0}$ forms a basis of neighbourhoods of $x_m(y_0)$. Let $U_\epsilon$ be a neighbourhood of $y_0$ included in $U$ such that for all $y\in U_\epsilon$ and all $x \in K$,
$$|F_y(x) - F_{y_0}(x)| < \frac{\epsilon}{2}~.$$
($U_\epsilon$ exists because the map $y \mapsto F_y$ is continuous for the compact-open topology.)
Since $x_m(y_0)$ is obviously in $K$, we have for all $y\in U_\epsilon$,
$$F_y(x_m(y_0)) < m(y_0) + \frac{\epsilon}{2}~,$$
hence the minimum value $m(y)$ of $F_y$ is smaller than $m(y_0) + \frac{\epsilon}{2}$. In particular, for $\epsilon < 2$, this minimum is achieved in $K$ (since outside $K$, we have $F_y \geq m(y_0) +1$). We thus have $x_m(y) \in K$, from which we deduce
\begin{displaymath}
F_{y_0}(x_m(y)) < F_y(x_m(y)) + \frac{\epsilon}{2} =  m(y) + \frac{\epsilon}{2} < m(y_0) + \epsilon~.
\end{displaymath}
We have thus proved that $x_m(y) \in V_\epsilon$ for all $y\in U_\epsilon$. Since $\left(V_\epsilon \right)_{\epsilon > 0}$ is a basis of neighbourhoods of $x_m(y_0)$, this proves that $y \mapsto x_m(y)$ is continuous at $y_0$.
\end{proof}

To prove the continuity of $\PPsi^{-1}$, we can apply proposition \ref{p:UniformProp} to the family $\F_{j,\rho}$ of functions on $\Teich(S)$ depending on the parameter $(j,\rho) \in \Dom(S,\Isom(M))$. The continuity of $(j,\rho) \mapsto \F_{j,\rho}$ comes from proposition \ref{p:ContinuityEPhi}. The only thing we need to check is thus that the family $$\left( \F_{j,\rho} \right)_{(j,\rho) \in \Dom(S, \Isom(M))}$$ is locally uniformly proper. But this follows easily from the continuity of the minimal Lipschitz constant. 
Indeed, let $(j_0,\rho_0)$ be a point in $\Dom(S,\Isom(M))$. We thus have $\Lip(j_0,\rho_0) < 1$. By continuity of the function $\Lip$, there exists a neighbourhood $U$ of $(j_0,\rho_0)$ and a $\lambda < 1$ such that for all $(j,\rho) \in U$, we have $\Lip(j,\rho) \leq \lambda$. By lemma \ref{l:properness}, we thus have
$$\F_{j,\rho}(Y) \geq \left( \frac{1- \lambda}{2}\right) \E(Y,j)$$
for all $(j,\rho) \in U$ and all $Y\in \Teich(S)$.
Since the function $Y \mapsto \E(Y,j)$ is proper by theorem~\ref{t:PropernessEnergy}, we obtain that the family $\left( \F_{j,\rho} \right)$ is uniformly proper on $U$. Hence it is locally uniformly proper.

By proposition~\ref{p:UniformProp}, we deduce that the unique minimum of $\F_{j,\rho}$ varies continuously with $(j,\rho)$. Since this minimum is precisely $\PPsi^{-1}(j,\rho)$, we proved that $\PPsi^{-1}$ is continuous. This concludes the proof of theorem \ref{t:MonThm}.

\appendix
\newpage

\appendixpage

\addappheadtotoc

This appendix is devoted to the proof of the following proposition (cf. proposition \ref{p:ContinuityLip}):
\begin{propstar}
Let $S$ be a closed oriented surface of negative Euler characteristic and $M$ a Riemannian $\CAT(-1)$ space. Then the function
\[ \function{\Lip}{\Teich(S)\times \Rep(S,\Isom(M))}{\R_+}{(j,\rho)}{\inf \left\{ \lambda \mid \exists f: \H^2 \to M\textrm{ $(j,\rho)$-equivariant and $\lambda$-Lipschitz} \right\} }\]
is continuous.
\end{propstar}

The proof will proceed in three steps. First, we will use the continuity of Thurston's asymmetric distance on $\Teich(S)$ to show that $\Lip(j,\rho)$ is continuous in $j$. Then we will prove upper semi-continuity in $\rho$ with a general argument that allows to deform a $(j,\rho)$-equivariant map into a $(j,\rho')$-equivariant map for $\rho'$ close to $\rho$, without increasing too much the Lipschitz constant. Finally, we will prove lower semi-continuity by showing that a sequence of equi-Lipschitz and $(j,\rho_n)$-equivariant maps converges to a $(j,\rho)$-equivariant map when $\rho_n$ converges to $\rho$ as soon as $\rho$ is not parabolic. The tricky part of the proof will be to deal with lower semi-continuity at parabolic representations. This will require some preliminary work on isometry groups of $\CAT(-1)$-spaces.

The first section of this appendix will therefore be a miscellaneous recall of classical facts about $\CAT(-1)$ spaces and their isometries. In the second section, we will characterize parabolic representation by the additivity of their length spectrum. Finally, we will prove proposition \ref{p:ContinuityLip} in the third section.

\section{$\CAT(-1)$-spaces and their isometries}

Recall that a \emph{length space} is a metric space in which the distance between two points is defined as the infimum of the lengths of paths joining these two points. If $x_1$, $x_2$ and $x_3$ are three points of a length space, a triangle with vertices $x_1$, $x_2$ and $x_3$ (denoted $ \tau(x_1,x_2,x_3)$) is the union of three geodesic segments joining $x_1$ with $x_2$, $x_2$ with $x_3$ and $x_3$ with $x_1$.

Let $(M,g_M)$ be a complete simply connected Riemannian manifold of sectional curvature bounded above by $-1$. Then $M$ is in particular a $\CAT(-1)$ space, in the following sense:

\begin{definition} \label{d:CAT(-1)}
A length space $(M,d)$ is $\CAT(-1)$ if for any triangle $\tau(x_1,x_2,x_3)$ in $M$, if $\tau(\bar{x}_1,\bar{x}_2,\bar{x}_3)$ is a triangle in $\H^2$ whose edges have the same length as those of $\tau(x_1,x_2,x_3)$, then the map from $\tau(\bar{x}_1,\bar{x}_2,\bar{x}_3)$ to $M$ that sends $\bar{x}_i$ to $x_i$ and that is geodesic on each edge is $1$-Lipschitz.
\end{definition}

Lots of properties of the hyperbolic plane can be extended to any $\CAT(-1)$ space by simple triangle comparison. We did not find a reference specific to $\CAT(-1)$ geometry, but the facts we state here are usually particular cases of more general properties of Gromov-hyperbolic spaces or $\CAT(0)$ spaces. One can consult \cite{GhysDeLaHarpe} or \cite{BridsonHaefliger} for more information.

\subsection*{Hyperbolic angles.}
Let $M$ be a $\CAT(-1)$ space and $x$, $y$, $z$ three points in $M$. We will denote by
\[ \anglehyp{y,x,z} \]
the angle at the vertex $\bar{x}$ of a hyperbolic triangle $\tau( \bar{x}, \bar{y}, \bar{z} )$ whose edges have the same length as those of $\tau(x,y,z)$.

\begin{prop} \label{p:InegaliteTriangulaireAngles}
Hyperbolic angles satisfy a triangular inequality: if $x$, $y$, $z$ and $t$ are four points in $M$, then
\[ \anglehyp{y,x,t} \leq \anglehyp{y,x,z} + \anglehyp{z,x,t} ~ .\]
\end{prop}

\begin{prop} \label{p:AnglesGrandTriangle}
Let $\tau_n$ be a sequence of triangles in $M$ such that the lengths of the edges of $\tau_n$ go to $+\infty$ and such that the hyperbolic angles at the vertices converge. Then at least two of these angles go to $0$.
\end{prop}

\subsection*{Boundary at infinity} \label{as:CAT(-1)}
A $\CAT(-1)$ space $M$ is in particular \emph{hyperbolic} in the sense of Gromov. One can thus define the \emph{boundary at infinity} of $M$ as the quotient of the space of geodesic rays going at speed $1$ by the equivalence relation that identifies two rays $\gamma_1$ and $\gamma_2$ if and only if the distance between $\gamma_1(t)$ and $\gamma_2(t)$ is bounded when $t$ goes to $+\infty$. This boundary at infinity, denoted $\partial_\infty M$, gives a compactification of $M$. When $M$ is a Riemannian $\CAT(-1)$ space, $M \cup \partial_\infty M$ is homeomorphic to a closed ball. Every geodesic $\gamma: \R \to M$ converges to a point $\gamma_+$ in $\partial_\infty M$ when $t$ goes to $+\infty$ and a point $\gamma_-$ when $t$ goes to $-\infty$. Conversely, any two points in $\partial_\infty M$ are the extremities of a unique geodesic in $M$.

The hyperbolic angles extend to the boundary, meaning that if $x$ is a point in $M$ and $(y_n)$ and $(z_n)$ two sequences in $M$ converging respectively to $p$ and $q$ in $\partial_\infty M$, then $\anglehyp{y_n, x, z_n}$ converges to a value that only depends on $x$, $p$ and $q$. We will denote it $\angle_\infty (p,x,q)$. The function
\[\function{d_x}{\partial_\infty M \times \partial_\infty M}{\R_+}{(p,q)}{\angle_\infty(p,x,q)}\]
is a distance on $\partial_\infty M$.

\subsection*{Busemann functions.}
Fix a point $p$ in $\partial_\infty M$ and let $x_0$ and $x$ be two points in $M$. Consider the geodesic ray $\gamma: \R_+ \to M$ starting at $x$ and pointing toward $p$. Then 
\[d(x_0, \gamma(t)) - t\] converges to a number denoted $B_{p,x_0}(x)$. By triangular inequality, we clearly have
\begin{equation} \label{eq:BusemanVSDistance}
|B_{p,x_0}(x) | \leq d(x_0,x) ~ .
\end{equation}
The function $B_{p,x_0}$ is called a \emph{Busemann function}. This function only depends on $x_0$ up to a translation. Indeed, given $x_0$, $x_1$ and $x$ in $M$, we have
\begin{equation} \label{eq:IdentiteBuseman}
B_{p,x_0}(x) = B_{p,x_0}(x_1) + B_{p,x_1}(x) ~ .
\end{equation}

The level sets of the Busemann function $B_{p,x_0}$ are called \emph{horospheres} tangent to $p$. When $M$ is a Riemannian  $\CAT(-1)$ space of dimension $n$, horospheres are smooth hypersurfaces diffeomorphic to $\R^{n-1}$. They form a foliation orthogonal to the $1$-dimensional foliation consisting of geodesics pointing toward $p$.

For every $t \in \R_+$, we will denote by $F_{p,t}$ the map sending a point $x$ in $M$ to the point $\gamma_{x,p}(t)$, where $\gamma_{x,p}$ is the geodesic ray starting at $x$ and pointing toward $p$. The map $F_{p,t}$ preserves the foliation of $M$ by horospheres tangent to $p$ and induces a surjection from $B_{p,x_0}^{-1}(a)$ to $B_{p,x_0}^{-1}(a+t)$ for any $a\in \R$ (figure \ref{fig:ContractionBusemann}). When $M$ is Riemannian, this surjection is a diffeomorphism.

\begin{figure} 
\begin{center}
\includegraphics[width=12cm]{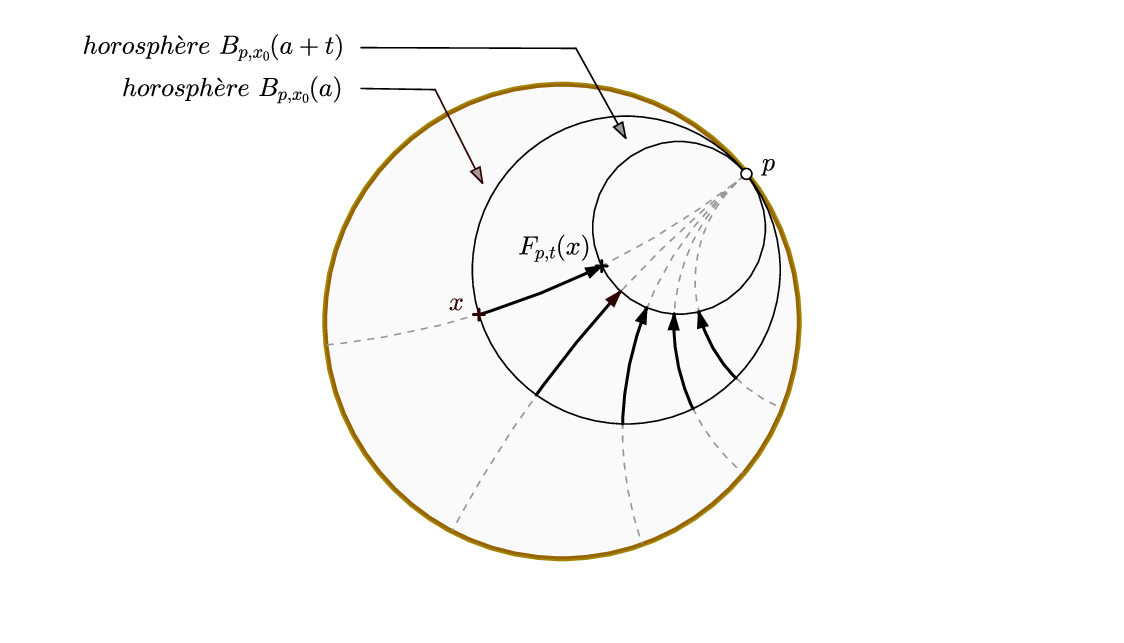}
\caption{The map $F_{p,t}$ sending the horosphere $B_{p,x_0}^{-1}(a)$ to the horosphere $B_{p,x_0}^{-1}(a+t)$.}
\label{fig:ContractionBusemann}
\end{center}
\end{figure}

\begin{prop} \label{p:ContractionBuseman}
The function $F_{p,t}$ is $1$-Lipschitz and satisfies the following inequality:
\[ |B_{p,x}(y)| \leq d\left( F_{p,t}(x), F_{p,t}(y)\right) \leq | B_{p,x}(y) | + e^{-t}  d(x,y) ~ .\]
In particular, the distance between $F_{p,t}(x)$ and $F_{p,t}(y)$ converges to $|B_{p,x}(y)|$ when $t$ goes to $+\infty$.
\end{prop}

\subsection*{Isometries.}
Isometries of $M$ extend to homeomorphisms of $M \cup \partial_\infty M$. They can have three different types of dynamics, depending on their translation length.

Let $g$ be an isometry of $M$. If $l(g)= 0$, then $g$ fixes either a point in $M$ --  in that case $g$ is called \emph{elliptic} -- or a  single point in $\partial_\infty M$ --  it is then called \emph{parabolic}. An elliptic isometry preseves the distance to a fixed point and has bounded orbits, while a parabolic isometry preserves the horospheres tangent to its fixed point.

When $l(g) > 0$, the isometry $g$ is called \emph{hyperbolic}. In that case, the set of points $x$ in $M$ for which $d(x,g \cdot x) = l(g)$ forms a geodesic in $M$ called the axis of $g$. This axis is preserved by $g$, and $g$ acts along its axis as a translation of length $l(g)$. The extremities of the axis in $\partial_\infty M$ are fixed by $g$.

\begin{prop} \label{p:IsometrieBuseman}
Let $g$ be an isometry of $M$ fixing a point $p$ in $\partial_\infty M$. Then the number
\[ B_{p,x}(g \cdot x) \]
does not depend on $x$ and verifies
\[ \left| B_{p,x}(g\cdot x) \right| = l(g) ~ .\]
\end{prop}



\subsection*{Properties of hyperbolic isometries.}
Let us finally state a few properties of hyperbolic isometries that will be useful in the next section.
The first one asserts that elements of $M$ that ``do not move too much'' under the action of a hyperbolic isometry $g$ must be close to the axis of $g$.
\begin{prop} \label{p:DeplacementIsomHyperbolique}
For any $\epsilon>0$ and any constant $k$, there exists a constant $C(\epsilon,k)$ such that if $g$ is an isometry of $M$ such that $l(g) \geq \epsilon$ and $x$ a point in $M$ such that 
\[d(x,g\cdot x) \leq k ~ ,\]
then
\[d(x, \axe(g))\leq C(\epsilon,k) ~ .\]
\end{prop}
\begin{rmk}
Though we don't need it, we could give an explicit value of $C(\epsilon,k)$ by comparing with the case of $\H^2$.
\end{rmk}

The next proposition states that hyperbolic isometries have ``north-south'' dynamics in the following sense:
\begin{prop} \label{p:HyperboliqueDynamiqueNS}
Let $g$ be a hyperbolic isometry of $M$ and $g_-$, $g_+$ the extremities of its axis (oriented in the direction in which $g$ translates). Then there exist sequences of compact subsets $(U_n^-)$ and $(U_n^+)$ in $M \cup \partial_\infty M$ such that
\[\bigcap_{n \in \N} U_n^- = g_- ~ , \]
\[\bigcap_{n\in \N} U_n^+ = g_+ ~ , \]
and such that $g^n$ sends the complement of $U_n^-$ into $U_n^+$.
\end{prop}
What's more, this type of dynamics caracterizes hyperbolic isometries:
\begin{prop} \label{p:CaracterisationDynamiqueNS}
Let $g_n$ be a sequence of isometries in $M$. Assume that there exists two distinct points $p_-$ and $p_+$ in $\partial_\infty M$ and two sequences of compact subsets $(U_n^-)$ and $(U_n^+)$ in $M \cup \partial_\infty M$ such that
\[\bigcap_{n \in \N} U_n^- = p_- ~ , \]
\[\bigcap_{n\in \N} U_n^+ = p_+ ~ , \]
and such that $g_n$ sends the complement of $U_n^-$ into $U_n^+$. Then the translation length of $g_n$ goes to $+\infty$ when $n$ goes to $+\infty$. In particular, $g_n$ is hyperbolic for $n$ big enough.
\end{prop}

Let us finish with a proposition that characterizes hyperbolic isometries by looking only at three points in an orbit.
\begin{prop} \label{p:CaracterisationIsomHyperboliqueAngles}
For any $\epsilon > 0$, there is a constant $C$ such that for any $g$ in $\Isom(M)$, if there exists a point $x$ such that
\[d(x,g\cdot x) \geq C \]
and 
\[\anglehyp{g^{-1} \cdot x, x, g \cdot x} \geq \epsilon ~ ,\]
then $g$ is hyperbolic.
\end{prop}

Again, we could specify the constant $C$ by comparing with the case of $\H^2$.

\section{Parabolic representations}

Recall that a representation $\rho: \pi_1(S) \to \Isom(M)$ is called \emph{parabolic} if it fixes a point in $\partial_\infty M$. The purpose of this section is to prove two lemmas that will allow us to show lower semi-continuity of $\Lip(j, \, \cdot \,)$ at a parabolic representation in the next section. The first lemma characterises parabolic representations by the additivity of their length spectrum, while the second one relates the minimal Lipschitz constant to this length spectrum.

\begin{lem} \label{l:CaracterisationsRepresentationsReductibles}
Let $\Gamma$ be a subgroup of $\Isom(M)$. The following are equivalent:
\begin{itemize}
\item[(i)] $\Gamma$ fixes a point in $M \cup \partial_\infty M$,
\item[(ii)] there exists a morphism $m$ from $\Gamma$ to $\R$ such that for all $g\in \Gamma$,
\[l(g) = |m(g)| ~ . \]
\end{itemize}
\end{lem}

\begin{lem} \label{l:ConstanteLipschitzRepReductible}
Let $j$ be a Fuchsian representation from $\pi_1(S)$ to $\Isom(\H^2)$ and $\rho: \pi_1(S) \to \Isom(M)$ a parabolic representation. Let $m:\pi_1(S) \to \R$ be a morphism such that
\[ l(\rho(\gamma)) = |m(\gamma)| \]
for all $\gamma \in \pi_1(S)$.
Denote by $\Lip(j,m)$ the minimal Lipschitz constant of a $(j,m)$-equivariant map from $\H^2$ to $\R$. Then
\[ \Lip(j,\rho) = \Lip(j,m)~ .\]
\end{lem}

\begin{proof}[Proof of lemma \ref{l:CaracterisationsRepresentationsReductibles}]{\

}

$(i) \Rightarrow (ii)$.\\
If $\Gamma$ has a fixed point in $M$, then $l(g) = 0$ for any $g\in \Gamma$ and $m$ is the trivial morphism. Otherwise, let $p \in \partial_\infty M$ be a fixed point of $\Gamma$. Choose a base point $x_0$ in $M$ and set
\[m(g) = B_{p,x_0}(g \cdot x_0)\]
for all $g \in \Gamma$. According to proposition \ref{p:IsometrieBuseman}, we have
\[| m(g) | = l(g) ~ .\]

According to equation (\ref{eq:IdentiteBuseman}), for any $g$, $h$ in $\Gamma$, we have
\begin{eqnarray*}
m(gh) & = & B_{p,x_0}(gh \cdot x_0) \\
\ & = & B_{p,x_0}(h \cdot x_0) + B_{p,h \cdot x_0}(gh \cdot x_0) ~ .
\end{eqnarray*}
By proposition \ref{p:IsometrieBuseman}, we get
\[B_{p,h\cdot x_0}(gh \cdot x_0) = B_{p,x_0}(g \cdot x_0)\]
hence
\begin{eqnarray*}
m(gh) & = & B_{p,x_0}(g \cdot x_0) + B_{p,x_0}(h \cdot x_0) \\
\ & = & m(g) + m(h) ~ .
\end{eqnarray*}
The function $m$ is thus a morphism from $\Gamma$ to $\R$.\\

$(ii)\Rightarrow (i)$.\\
In order to prove the converse, we have to consider two distinct cases.\\

$\bullet$ { \it Case 1: $m \centernot{\equiv} 0. \quad$}
In that case, $\Gamma$ contains hyperbolic elements. Let us prove that all the axes of those elements have a common extremity. It is then easy to deduce that this common extremity is fixed by $\Gamma$.

Assume that there exist $g$ and $h$ two hyperbolic isometries in $\Gamma$ whose axes have all extremities disjoint. Up to replacing $g$ and $h$ by some powers, we can assume that the group generated by $g$ and $h$ has a Schottky-type dynamics (figure \ref{fig:DynamiqueSchottky}).The general idea is then that the orbit of $x$ under this group must be quasi-isometric to a free group with two generators equipped with its Cayley metric, which contradicts the fact that the translation lengths of the elements of this group are given by a morphism with values in $\R$. 

To be more precise, we can apply proposition \ref{p:HyperboliqueDynamiqueNS} to $g$ and $h$ and consider $(U_n^+)$, $(U_n^-)$, $(V_k^+)$ and $(V_k^-)$ four sequences of compact subsets in $M\cup \partial_\infty M$ such that $\bigcap U_n^+ = g^+$, $\bigcap U_n^- = g^-$, $\bigcap V_k^+ = h^+$ and $\bigcap V_k^- = h^-$, and such that $g^n$ sends the complement of $U_n^-$ into $U_n^+$, $g^{-n}$ sends the complement of $U_n^+$ into $U_n^-$, $h^k$ sends the complement of $V_k^-$ into $V_k^+$ and $h^{-k}$ sends the complement of $V_k^+$ into $V_k^-$.

\begin{figure} 
\begin{center}
\includegraphics[width=8cm]{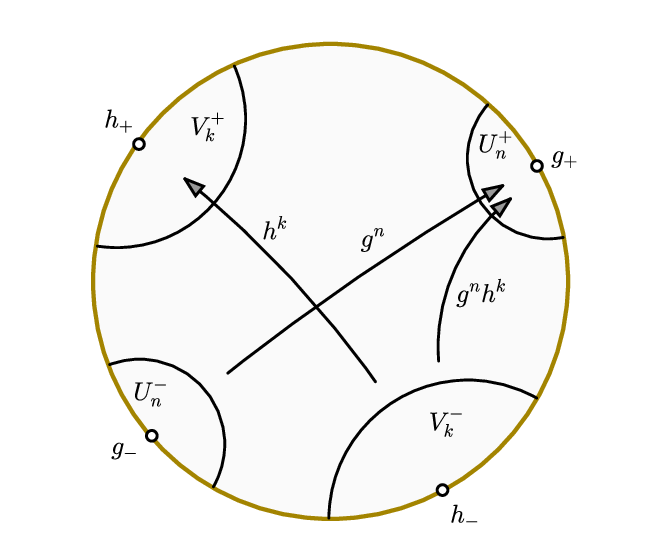}
\caption{Schottky-type dynamics.}
\label{fig:DynamiqueSchottky}
\end{center}
\end{figure}

For $n$ and $k$ large enough, the sets $U_n^+$, $U_n^-$, $V_k^+$ and $V_k^-$ are pairwise disjoint and the isometry $g^n h^k$ sends the complement of $V_k^-$ into $U_n^+$. Since $\bigcap U_n^+ = g_+$ and $\bigcap V_k^- = h_-$, we obtain from proposition \ref{p:CaracterisationDynamiqueNS} that the translation length of $g^n h^k$ goes to $+\infty$ when $n$ and $k$ go simulatenously to $+\infty$. Replacing $h$ by $h^{-1}$, one would obtain that the translation length of $g^n h^{-k}$ goes to infinity when $n$ and $k$ go to $+\infty$. 

On the other side, we have
\[l(g^n h^k) = | n \ m(g) + k\ m(h) | ~,\]
and since $m(g)$ and $m(h)$ are non-zero, one can find, for any integer $n$, an integer $k$ such that
\[l(g^n h^k) \leq |m(h)| ~. \]
This contradicts the fact that $l(g^n h^k)$ goes to infinity when $n$ and $k$ go to infinity.

This proves that the axes of any two hyperbolic elements in $\Gamma$ have a common extremity. One can easily deduce from this fact that all those axes have a common extremity which is fixed by $\Gamma$.

$\bullet$ {\it Case 2:  $m \equiv 0. \quad $} This is the very specific case where no element of $\Gamma$ is hyperbolic. Let $x$ be a point in $M$. If the orbit of $x$ under $\Gamma$ is bounded, then $\Gamma$ stabilizes this orbit and fixes the center of the smallest closed ball containing this orbit. Otherwise, let us shaw that the closure of $\Gamma \cdot x$ contains a unique point in $\partial_\infty M$. This point will therefore be fixed by $\Gamma$.  The argument that we use is a variation on a classical argument in the study of Gromov-hyperbolic spaces (see for instance \cite{GhysDeLaHarpe}).

Let us assume by contradiction that the closure of $\Gamma \cdot x$ contains two distinct points $p$ and $p'$ of $\partial_\infty M$. Let $(g_n)$ and $(h_n)$ be two sequences of elements in $\Gamma$ such that \[g_n \cdot x \tend{n\to +\infty} p\] and \[h_n \cdot x \tend{n \to + \infty} p'~.\] Since $p$ and $p'$are distinct, the angle \[\alpha_n =\anglehyp{g_n \cdot x, x,  h_n \cdot x}\] converges to $\alpha = \angle_\infty(p,x,p')$ which is non-zero.

Since $g_n$ is not hyperbolic and since $d(x, g_n \cdot x)$ goes to infinity, we obtain from proposition \ref{p:CaracterisationIsomHyperboliqueAngles} that the angle
\[\anglehyp{g_n \cdot x, x, g_n^{-1} \cdot x }\]
goes to $0$ when $n$ goes to infinity.

What's more, the lengths of the triangle $\tau(x, h_n \cdot x, g_n \cdot x)$ all go to $+\infty$ and, from proposition \ref{p:AnglesGrandTriangle}, we deduce that the angle $\anglehyp{x, h_n \cdot x, g_n \cdot x }$ goes to $0$ when $n$ goes to $+ \infty$. Applying the isometry $g_n^{-1}$, we get that the angle $\anglehyp{g_n^{-1} \cdot x, x, g_n^{-1} h_n \cdot x }$ also goes to $0$. By triangular inequality, we thus know that
\[\anglehyp{g_n^{-1} h_n \cdot x , x, g_n \cdot x} \]
goes to $0$ (the figure \ref{fig:IsometrieHyperbolique} helps visualizing those angles). Reversing the roles of $g_n$ and $h_n$, one would prove similarly that
\[\anglehyp{h_n^{-1} g_n \cdot x , x, h_n \cdot x}\]
goes to $0$. We eventually obtain that the angle
\[\anglehyp{g_n^{-1} h_n \cdot x, x , h_n^{-1} g_n \cdot x } \]
goes to $\alpha \neq 0$. Since the distance $d(x, h_n^{-1} g_n \cdot x)$ grows to infinity, we deduce from proposition \ref{p:CaracterisationIsomHyperboliqueAngles} that $h_n^{-1} g_n$ is a hyperbolic isometry for $n$ large enough, which contradicts the hypothesis.

\begin{figure} 
\begin{center}
\includegraphics[width=8cm]{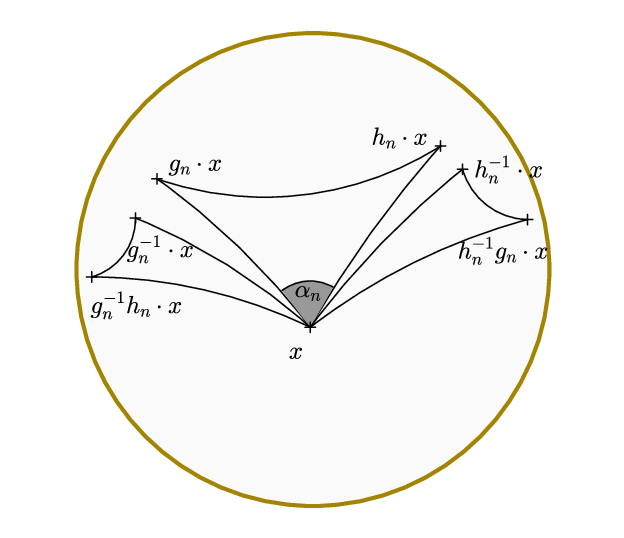}
\caption{The isometry $h_n^{-1} g_n$ is hyperbolic.}
\label{fig:IsometrieHyperbolique}
\end{center}
\end{figure}

Therefore, the closure of $\Gamma \cdot x$ contains a unique point in $\partial_\infty M$ which is fixed by $\Gamma$. This concludes the proof of lemma \ref{l:CaracterisationsRepresentationsReductibles}.
\end{proof}

\begin{proof}[Proof of lemma \ref{l:ConstanteLipschitzRepReductible}]
Let $j: \pi_1(S) \to \Isom(\H^2)$ be a Fuchsian representation and $\rho: \pi_1(S) \to \Isom(M)$ a parabolic representation. According to lemma \ref{l:CaracterisationsRepresentationsReductibles}, there exists a morphism $m: \pi_1(S) \to \R$ such that for all $g \in \pi_1(S)$, 
\[l(g) = |m(g)| ~ .\]
Let us prove that
\[\Lip(j,\rho) = \Lip(j,m) ~ .\]

Let $p \in \partial_\infty M$ be the fixed point of $\rho$. Recall that the morphism $m$ is given by 
\[m(\gamma) = B_{p,x}(\rho(\gamma) \cdot x) ~, \]
for any point $x \in M$.
Consider first a $(j,\rho)$-equivariant and $C$-Lipschitz map $f: \H^2 \to M$. Fix a base point $x_0 \in M$ and set
\[ \function{\bar{f}}{\H^2}{\R}{x}{B_{p,x_0}(f(x)) ~ .} \]
We easily check that $\bar{f}$ is $(j,m)$-equivariant. Moreover, we have, for any $x$, $y \in \H^2$,
\begin{eqnarray*}
|\bar{f}(y) - \bar{f}(x)| & = & | B_{p,x_0}(f(y)) - B_{p,x_0}(f(x)) | \\
\ & = & | B_{p,f(x)}(f(y)) | \\
\ & \leq & d(f(x), f(y)) ~ ,
\end{eqnarray*}
hence $\bar{f}$ is also $C$-Lipschitz. We obtain that
\[ \Lip(j,m) \leq \Lip(j,\rho) ~ . \\ \]

Let us now prove the converse inequality. Let $\bar{f}$ be a $(j,m)$-equivariant and $C$-Lipschitz function from $\H^2$ to $\R$.
Let $S\times_\rho M$ denote the flat $M$-bundle 
\[\H^2 \times M / (j,\rho)(\pi_1(S)) ~ , \]
and $S \times_m \R$ the flat $\R$-bundle
\[ \H^2 \times \R / (j,m)(\pi_1(S)) ~ .\]
The map from $\H^2 \times M$ to $\H^2 \times \R$ sending $(x,y)$ to $(x,B_{p,x_0}(y))$ induces a fibration
\[\pi : S\times_\rho M \to S \times_m \R ~, \]
whose fibers are horospheres tangent to $p$. Since those horospheres are contractible, $\pi$ admits a smooth section that allows to lift a $(j,m)$-equivariant map $\bar{f}: \H^2 \to \R$ to a $(j,\rho)$-equivariant map $f: \H^2 \to M$ such that
\[\bar{f}(x) = B_{p,x_0}(f(x)) \]
for all $x\in \H^2$. 

This map $f$ is $C'$-Lipschitz for some constant $C'$, but $C'$ is a priori greater than $C$. However, since $p$ is fixed by $\rho$, the map $F_{p,t}$ commutes with $\rho(\gamma)$ for all $\gamma \in \pi_1(S)$. Consequently, the map $F_{p,t} \circ f$ is another $(j,\rho)$-equivariant map. Now, by proposition \ref{p:ContractionBuseman}, for all $x$ and $y$ in $\H^2$ we have
\begin{eqnarray*}
d\left( F_{p,t}\circ f(x) , F_{p,t} \circ f(y) \right) & \leq & |B_{p,f(x)}(f(y))| + e^{-t} d(f(x), f(y)) \\
\ & \leq & |\bar{f}(y) - \bar{f}(x)| + e^{-t} C' \ d(x,y) \\
\ & \leq &  \left( C+ C'  e^{-t}\right) d(x,y) ~ .
\end{eqnarray*}
The map $F_{p,t} \circ f$ is thus $(C+ C' e^{-t} )$-Lipschitz. Since this is true for all $C > \Lip(j,m)$ and all $t\in \R_+$, we obtain \[\Lip(j,\rho) \leq \Lip(j,m)~.\]
This concludes the proof of lemma \ref{l:ConstanteLipschitzRepReductible}.
\end{proof}

\section{Continuity of the minimal Lipschitz constant}

We are now able to prove the continuity of the minimal Lipschitz constant. The arguments are not new and the proof essentially follows the one of Gu\'eritaud--Kassel \cite{GueritaudKassel} (where they assume that the target space is the hyperbolic space). As we said before, the tricky part will be to prove lower semi-continuity at a parabolic representation $\rho$.

\subsection{Continuity with respect to $j$}

Let us start by fixing a representation $\rho \in \Rep(S,\Isom(M))$. The continuity of the map $j \to \Lip(j,\rho)$ easily follows from the continuity of Thurston's asymmetric distance on $\Teich(S)$ \cite{Thurston86}. 
Let $j_0:\pi_1(S) \to \Isom(\H^2)$ be a Fuchsian representation and fix $\epsilon >0$. By continuity of Thurston asymetric distance, there exists a neighbourhood $U_\epsilon$ of $j_0$ such that for all $j \in U_\epsilon$, we have $\Lip(j_0,j) \leq 1+ \epsilon$ and $\Lip(j,j_0) \leq 1+ \epsilon$. Let $j$ be a point in $U_\epsilon$ and $h :\H^2 \to \H^2$ a $(j,j_0)$-equivariant and $(1+2\epsilon)$-Lipschitz map. Let otherwise $f : \H^2 \to M$ be a $(j_0,\rho)$-equivariant and $\left(\Lip(j_0,\rho) + \epsilon \right)$-Lipschitz map. Then the map 
$$f \circ h: \H^2 \to M$$
is $(j,\rho)$-equivariant and $(1+\epsilon)\left(\Lip(X,\rho)+ \epsilon\right)$-Lipschitz. We thus have
\[\Lip(j,\rho) \leq (1+\epsilon)\left(\Lip(j_0,\rho) + \epsilon \right)~.\]
Switching $j_0$ and $j$, we get in the same manner that
\[\Lip(j_0,\rho) \leq (1+\epsilon)\left(\Lip(j,\rho) + \epsilon \right)~.\]
We eventually obtain
\[\Lip(j_0,\rho) - \mathrm{Cste} \ \epsilon \leq \Lip(j,\rho) \leq \Lip(j_0,\rho) + \mathrm{Cste} \ \epsilon\]
for all $j\in U_\epsilon$, which proves the continuity of $j \mapsto \Lip(j,\rho)$ at $j_0$.\\

\subsection{Upper semi-continuity in $\rho$}

Since the group of isometries of a Riemannian manifold is a Lie group, the space $\Hom(S,\Isom(M))$ possesses a structure of real algebraic variety. In particular it is arcwise connected and we just have to show that if $(\rho_t)_{t\in [0,\epsilon[}$ is a continuous family of representations of $\pi_1(S)$ into $\Isom(M)$, then
\[ \limsup_{t\to 0} \Lip(j,\rho_t) \leq \Lip(j,\rho_0) ~ .\]

To this purpose, one can see $\rho_t$ as the monodromy of a flat connection $\nabla_t$ on a fixed $M$-bundle $E$ over $S\simeq j(\pi_1(S)) \backslash \H^2$, varying continuously with $t$. Fix $C > \Lip(j,\rho_0)$ and let $f_0$ be a $(j,\rho_0)$-equivariant and $C$-Lipschitz map from $\H^2$ to $M$. The map $f_0$ induces a section $s_0$ of the flat bundle $(E,\nabla_0)$, which in turn can be seen as a section $s_t$ of $(E,\nabla_t)$ associated to a $(j,\rho_t)$-equivariant map $f_t: \H^2\to M$. The continuity of $t\to \nabla_t$ and the compacity of $S$ ensure that $f_t$ is $\left( C + \alpha(t)\right)$-Lipschitz for some function $\alpha(t)$ going to $0$ when $t$ goes to $0$. We deduce that
\[\limsup_{t\to 0} \Lip(j,\rho_t) \leq C\]
For any $C> \Lip(j,\rho_0)$ and thus that
\[ \limsup_{t\to 0} \Lip(j,\rho_t) \leq \Lip(j,\rho_0) ~ .\]
This proves upper semi-continuity.

\subsection{Lower semi-continuity in $\rho$} \label{ass:LowerContinuity}

To prove lower semi-continuity, consider a sequence $(\rho_n)$ of representations into $\Isom(M)$ converging to $\rho$. Suppose that there exists a sequence $f_n$ of $(j,\rho_n)$-equivariant and $C$-Lipschitz maps. Let us prove that
\[\Lip(j,\rho) \leq C ~ .\]
Fix a base point $x$ in $\H^2$. If $f_n(x)$ does not escape to infinity in $M$, then, by Ascoli's theorem, a sub-sequence of $(f_n)$ converges uniformly on every compact set to a $C$-Lipschitz and $\rho$-equivariant map. We conclude that $\Lip(j,\rho) \leq C$. If $f_n(x)$ escapes to infinity, we can assume that $f_n(x)$ converges to a point $p$ in $\partial_\infty M$.

\begin{prop} \label{p:ConditionAscoli}
Assume $f_n(x)$ converges to $p\in \partial_\infty M$. Then $p$ is fixed by $\rho(\gamma)$ for all $\gamma \in \pi_1(S)$.
\end{prop}

\begin{proof}
Fix $\gamma \in \pi_1(S)$. Note first that the sequence
\[d(f_n(x), \rho_n(\gamma) \cdot f_n(x))\]
is bounded. Indeed,
\begin{eqnarray*}
d(f_n(x), \rho_n(\gamma) \cdot f_n(x)) & = & d( f_n(x), f_n(j(\gamma) \cdot x))\\
\ & \leq & C \, d(x,j(\gamma) \cdot x)~.
\end{eqnarray*}
Since $p \in \partial_\infty M$ and since $f_n(x)$ converges to $p$, we can conclude that $\rho_n(\gamma) \cdot f_n(x)$ also converges to $p$. But $\rho_n(\gamma)$ converges uniformly to $\rho(\gamma)$ on $M\cup \partial_\infty M$. We thus have
\[ \rho(\gamma) \cdot p = p ~ .\]
\end{proof}

We are now left with the case where $\rho$ fixes a point $p$ in $\partial_\infty M$. According to lemma \ref{l:CaracterisationsRepresentationsReductibles}, there is a morphism $m : \pi_1(S) \to \R$ such that for all $\gamma \in \pi_1(S)$, 
\[l(\rho(\gamma)) = | m(\gamma) | ~ .\]

Assume first that $m \equiv 0$. Then, according to lemma \ref{l:ConstanteLipschitzRepReductible}, we have $\Lip(j,\rho)= \Lip(j,m) = 0$ and lower semi-continuity is trivial at $\rho$. Otherwise, let us prove that we can find a sequence $(g_n)$ of isometries of $M$ such that $g_n \circ f_n(x)$ is bounded. (Note that this is immediate when $M$ is homogeneous.) Let $\gamma \in \pi_1(S)$ be such that $\rho(\gamma)$ is hyperbolic. Let $A$ be the axis of $\rho(\gamma)$ and $l = l(\rho(\gamma))$. For $n$ big enough, $\rho_n(\gamma)$ is an isometry of axis $A_n$ and translation length $l_n$, and we have
\[ A_n \tend{n\to +\infty} A \]
and
\[l_n\tend{n\to +\infty} l ~ .\]

We saw (propostion \ref{p:ConditionAscoli}) that $d\left( f_n(x), \rho_n(\gamma) \cdot f_n(x) \right)$ remains bounded when $n$ goes to $+\infty$. From \ref{p:DeplacementIsomHyperbolique}, there exists a constant $B'$ such that for $n$ big enough, 
\[d(f_n(x), A_n) \leq B'.\]
Let $a_n \in A_n$ be such that $(a_n)$ converges to a point $a \in A$. Let $y_n$ be the orthogonal projection of $f_n(x)$ onto $A_n$. Since $\rho_n(\gamma)$ acts as a translation of length $l_n$ along $A_n$, there is an integer $k_n\in \Z$ such that
\[d_M \left(\rho_n(\gamma)^{k_n} \cdot y_n, a_n\right) \leq l_n~.\]
We eventually obtain
\[d_M \left(\rho_n(\gamma)^{k_n} \cdot f_n(x), a\right) \leq B' + l_n + d_M(a_n, a)~,\]
and since $d_M(a_n,a)$ and $l_n$ are bounded, $\rho_n(\gamma)^{k_n} \circ f_n(x)$ is bounded.

Now the map $\rho_n(\gamma)^{k_n} \circ f_n$ is $(j,\rho_n(\gamma)^{-k_n} \cdot \rho)$-equivariant and $C$-Lipschitz (where $\rho_n(\gamma)^{-k_n} \cdot \rho$ denotes the composition of $\rho$ with the conjugation by $\rho_n(\gamma)^{-k_n}$). By Ascoli's theorem, a sub-sequence of $\left( \rho_n(\gamma)^{k_n} \circ f_n \right)$ converges to $f'$, which is $C$-Lipschitz and $(j,\rho')$-equivariant for some representation $\rho'$. We thus have $\Lip(j,\rho') \leq C$. On the other side, for all $\gamma' \in \pi_1(S)$, we have
\begin{eqnarray*}
l(\rho'(\gamma')) & = & \lim_{n\to +\infty} l\left( \rho_n(\gamma)^{-k_n} \rho_n(\gamma') \rho_n(\gamma)^{k_n} \right)\\
\ & = & \lim_{n \to + \infty} l(\rho_n(\gamma'))\\
\ & = & l(\rho(\gamma'))\\
\ & = & |m(\gamma')|~ .
\end{eqnarray*}
Therefore, by lemma \ref{l:CaracterisationsRepresentationsReductibles}, the representation $\rho'$ also fixes a point in $M \cup \partial_\infty M$ and by lemma \ref{l:ConstanteLipschitzRepReductible} we have
\[\Lip(j,\rho') = \Lip(j,m) = \Lip(j,\rho)~ .\]
We thus obtain
\[\Lip(j,\rho) \leq C ~ .\]
This concludes the proof of the lower semi-continuity of the map $\rho \mapsto \Lip(j,\rho)$.

\bibliographystyle{plain}
\bibliography{mabiblio}

\end{document}